\documentclass[12pt]{amsart}
\usepackage{amsfonts,amsthm,amsmath,amssymb}
\allowdisplaybreaks
\usepackage[utf8]{inputenc}
\usepackage[T1]{fontenc}
\usepackage{indentfirst}
\usepackage{fullpage}
\usepackage{xcolor}
\usepackage{hyperref}
\usepackage{cleveref}

\numberwithin{equation}{section}
\date{}

\newcommand{\ddbar}{\sqrt{-1}\,\partial\overline\partial}

\newtheorem{theorem}{Theorem}
\newtheorem{propo}{Proposition}
\newtheorem{lema}{Lemma}

\crefname{propo}{Proposition}{Propositions}
\Crefname{propo}{Proposition}{Propositions}
\crefname{lema}{Lemma}{Lemmas}
\Crefname{lema}{Lemma}{Lemmas}

\theoremstyle{remark}
\newtheorem{remark}{Remark}

\theoremstyle{definition}
\newtheorem{example}{Example}

\hypersetup{
    colorlinks=true,
    linkcolor=blue,
    citecolor=green,
    filecolor=magenta,
    urlcolor=cyan,
}

\begin{document}
\raggedbottom

\title{Chern--Ricci flow on Kato surfaces}

\author[Daniele Angella]{Daniele Angella}
\address[Daniele Angella]{Dipartimento di Matematica e Informatica ``Ulisse Dini'', Universit\`a degli Studi di Firenze, viale Morgagni 67/A, 50134 Firenze, Italy}
\email{daniele.angella@unifi.it,
daniele.angella@gmail.com}

\author[Mauricio Corr\^ea]{Maur\'icio Corr\^ea}
\address[Mauricio Corr\^ea]{Dipartimento di Matematica,\newline Universit\`a degli Studi di Bari, \newline Via E. Orabona 4, I-70125, Bari, Italy
}
\email{mauricio.correa.mat@gmail.com, mauricio.barros@uniba.it}


\begin{abstract}
Let $S$ be a Kato surface and $D$ its maximal reduced divisor of rational curves. On $S\setminus D$ we construct Hermitian metrics which are flat along the leaves of the canonical foliation and study their evolution under the Chern--Ricci flow. For Enoki surfaces, we construct an immortal normalised solution which, on every compact sublevel of the natural exhaustion, converges in the Gromov--Hausdorff sense to the elliptic base endowed with an explicit flat metric. For Kato surfaces of intermediate type, the affine Green model and its finite-index extension give, under the assumption $0<\mu<2$, an explicit normalised solution which, on every compact Green block, collapses in the Gromov--Hausdorff sense to a circle. In the Enoki case the limiting area is $2\pi b_2(S)$; in the intermediate case the length of the limiting circle is determined by the Green and leafwise monodromies.
\end{abstract}

\maketitle

\section*{Introduction and motivation}

The classification of compact complex surfaces of class VII, namely those with Kodaira dimension $\mathrm{Kod}(X)=-\infty$ and first Betti number $b_1(X)=1$, is not yet complete.
One may assume that the surface is minimal, since every rational curve of self-intersection $-1$ can be contracted.
When $b_2(X)=0$, it is known from \cite{bogomolov-1976, bogomolov, li-yau-zheng0, li-yau-zheng, teleman} that these surfaces are either Hopf surfaces (namely compact complex surfaces with universal cover $\mathbb C^2\setminus \{0\}$) or Inoue surfaces \cite{inoue, bombieri}. Hopf surfaces admit a finite unramified cover diffeomorphic to $S^1\times S^3$.
Inoue surfaces, which are diffeomorphic to the total space of a bundle over $S^1$, contain no holomorphic curves but admit a holomorphic parabolic foliation without singularities, whose leaves are biholomorphic to either $\mathbb C$ or $\mathbb C\setminus\{0\}$ and are dense in the fibres of the smooth fibration.
Kato \cite{kato} generalised the first examples \cite{inoue-2, inoue-3} of class VII surfaces with $b_2(X)>0$ by iteratively blowing up the standard ball in $\mathbb C^2$ and then performing a holomorphic surgery.
They are characterised by the existence of a global spherical shell, namely an open subset $U \subset X$ biholomorphic to a neighbourhood of $S^3$ in $\mathbb C^2$ such that $X\setminus U$ is connected. Every Kato surface $X$ contains exactly $b_2(X)$ rational curves; conversely, every minimal compact complex surface in class VII containing exactly $b_2(X)>0$ rational curves is a Kato surface \cite{dloussky-oeljeklaus-toma}.
Kato surfaces are also degenerations of blown-up primary Hopf surfaces \cite{nakamura} and admit singular holomorphic foliations \cite{DO}.
The Global Spherical Shell conjecture remains open in general; substantial progress for small values of $b_2$ is due to Teleman \cite{teleman-Inv, teleman-Ann, teleman-PSPM, teleman-INdAM}. We refer to \cite{dloussky-Kato} for a detailed description of the geometry of Kato surfaces and to \cite{teleman-LNM} for a survey of the conjecture.

The ``optimistic conjecture'' of \cite[Section 4.6]{fang-tosatti-weinkove-zheng} aims to recover a global spherical shell from the limiting behaviour of the Chern--Ricci flow. The Chern--Ricci flow \cite{tosatti-weinkove} is the evolution equation for Hermitian metrics
$$ \frac{\partial \omega(t)}{\partial t} = -\operatorname{Ric}(\omega(t)), $$
where, locally, $\operatorname{Ric}(\omega(t))=-\sqrt{-1}\partial\overline\partial\log\det(g_{j\bar k}(t))$ is the Chern--Ricci form of $\omega(t)=\sqrt{-1}g_{j\bar k}(t)\,dz^j\wedge d\bar z^k$. When a positive $(1,1)$-form $\omega$ is regarded as a real Riemannian metric, we use the convention $g_\omega(u,v)=\omega(u,Jv)$,  where $J$ is the integrable almost complex structure. On a compact complex manifold the Chern--Ricci flow is equivalent to a scalar parabolic equation of Monge--Amp\`ere type. Standard parabolic theory therefore gives, for every initial metric $\omega(0)=\omega_0$, a unique solution on a maximal interval $[0,T)$, where $0<T\leq+\infty$. The maximal existence time is characterised as the supremum of those $t>0$ for which the class of $\omega_0-t\operatorname{Ric}(\omega_0)$ admits a positive representative up to a $\partial\overline\partial$-exact term; see \cite{tosatti-weinkove}. For general background on the Chern--Ricci flow, we refer to \cite{tosatti-weinkove, tosatti-weinkove-survey}; the surface case is treated in \cite{tosatti-weinkove-2} and the references therein.

For compact complex surfaces $S$, the Chern--Ricci flow may be non-collapsing at a finite time, meaning that $T<+\infty$ and the volume of $S$ with respect to $\omega(t)$ does not tend to zero as $t\to T$. In this case, as in the algebraic setting studied by Song--Weinkove \cite{song-weinkove}, the Chern--Ricci flow is expected to perform a canonical surgical contraction of finitely many $(-1)$-curves \cite{tosatti-weinkove-2, nie, to}.
The behaviour of the Chern--Ricci flow for minimal compact complex surfaces $S$ in class VII is known for some Hopf surfaces and for Inoue surfaces.

Consider a classical Hopf surface $H$ with parameters $\alpha,\beta\in\mathbb C$ satisfying $0<|\alpha|=|\beta|<1$, that is, a quotient of $\mathbb C^2\setminus\{0\}$ by the group generated by the diagonal contraction with eigenvalues $\alpha$ and $\beta$. It admits a standard metric $\omega_{\text{std}}$ whose lift to the universal cover is conformal to the flat K\"ahler metric; this K\"ahler cover is the K\"ahler cone over the Sasaki manifold $S^3$. The solution $\omega(t)$ of the Chern--Ricci flow with initial metric $\omega_{\text{std}}$ is explicit and collapses in finite time $T=\tfrac12$. As $t\to\tfrac12$, the metrics $\omega(t)$ smoothly converge to a non-negative $(1,1)$-form, whose kernel defines a smooth distribution whose iterated brackets generate the tangent space of the Sasaki $S^3$.
The metric spaces $(H, \omega(t))$ converge to $(S^1,d)$ in Gromov--Hausdorff topology, where $d$ is the distance on the circle $S^1\subset \mathbb R^2$ with radius $-\frac{\log|\alpha|}{\sqrt{2}\pi}$.
Thus the limiting geometry reflects both the differential and the holomorphic structures of $H$.
For primary linear Hopf surfaces, where possibly $|\alpha|\neq|\beta|$, related results were obtained in \cite{edwards}.

Inoue surfaces are classified in \cite{inoue} into three families, $S_M$, $S^+$, and $S^-$, according to their defining parameters.
Each is a fibre bundle over $S^1$, where the fibre is a $3$-dimensional torus in case $S_M$, and a compact quotient of the $3$-dimensional Heisenberg group in case $S^+$, while those of type $S^-$ admit an unramified double cover of type $S^+$.
They also carry a nonsingular holomorphic foliation whose leaves are biholomorphic to $\mathbb C$ in case $S_M$, and to $\mathbb C\setminus\{0\}$ in case $S^+$, and are dense in the fibres of the corresponding fibration.
The standard Tricerri--Vaisman metric, constructed in \cite{tricerri, vaisman}, is K\"ahler flat along the leaves of the foliation, and its lift to the universal cover is conformal to a K\"ahler metric.
Tosatti and Weinkove \cite{tosatti-weinkove-2} obtained an explicit solution with initial value the Tricerri--Vaisman metric. It exists for all $t\geq0$ and collapses as $t\to+\infty$. After renormalisation, the metrics $\omega(t)/t$ converge smoothly to a degenerate metric $\omega_\infty$ whose kernel defines the foliation. The metric spaces $(S, \omega(t)/t)$ likewise converge to $(S^1,d)$ in the Gromov--Hausdorff sense, where the radius of $d$ depends on the real eigenvalue or eigenvalues of the matrix defining $S$; see \cite[Section 5]{tosatti-weinkove-2}.
These convergence results extend to initial metrics which are $\partial\overline\partial$-deformations of metrics strongly flat along the leaves \cite{fang-tosatti-weinkove-zheng}, and in particular to Gauduchon metrics \cite{angella-tosatti}.

\bigskip

The purpose of this paper is to study the Chern--Ricci flow on Kato surfaces, complementing the results available for compact complex surfaces in class VII.
Throughout, if $S$ is a Kato surface, $D$ denotes its maximal reduced divisor of rational curves.
For a Gauduchon initial metric $\omega_0$ on $S$, the total volume evolves along the Chern--Ricci flow according to \cite{tosatti-weinkove}
$$\mathrm{Vol}(\omega(t))=\int_S\omega_0^2-4\pi t\int_S\omega_0\wedge c_1^{BC}(S)+4\pi^2t^2\int_S c_1^{BC}(S)^2.$$
Since $\int_S c_1^{BC}(S)^2=-b_2(S)<0$ by \cite[page 49]{dloussky-Kato}, the quadratic polynomial governing the total volume has a positive zero, so the cohomological volume obstruction occurs at a finite time. Let $C$ be an irreducible curve with $C^2<0$. By adjunction,
$$
K_S\cdot C+C^2=2p_a(C)-2,
$$
so that $c_1(S)\cdot C=2-2p_a(C)+C^2$. If $p_a(C)\geq1$, then $c_1(S)\cdot C\leq C^2<0$. If $p_a(C)=0$, then $C$ is a smooth rational curve; minimality excludes $C^2=-1$, hence $C^2\leq-2$ and again $c_1(S)\cdot C=2+C^2\leq0$. Therefore
$$
\frac{d}{dt}\int_C\omega(t)=-2\pi\,c_1(S)\cdot C\geq0.
$$
Thus the area of a negative curve cannot tend to zero along the flow. By \cite[Corollary 1.4]{tosatti-weinkove}, the maximal existence time for the Chern--Ricci flow on a compact complex surface with Gauduchon initial metric is characterised by the vanishing of either the total volume or the area of a curve of negative self-intersection. The latter alternative is excluded by the preceding inequality. Hence the maximal existence time is finite and the flow collapses in volume. We use the normalisation $\operatorname{Ric}(\omega)\in 2\pi c_1^{BC}(S)$.

On the open surface $S\setminus D$ we construct special Hermitian metrics together with explicit normalised Chern--Ricci solutions which collapse, along natural compact exhaustions, to an elliptic curve for Enoki surfaces and to a circle for Kato surfaces of intermediate type in the range $0<\mu<2$.

Recall that Kato surfaces are divided into several classes according to the configuration of the curves arising in their construction. Every Kato surface contains exactly $b:=b_2(S)$ rational curves $D_1,\dots,D_b$, some of which form a cycle $C$. The Dloussky number, defined by
$\sigma(S):=-\sum_{i=1}^b D_i^2$, satisfies $2b\leq\sigma(S)\leq 3b$, and the classes are distinguished by its value.
A Kato surface is called \emph{intermediate} if it contains a single cycle of rational curves together with at least one tree of rational curves meeting the cycle; equivalently, the Dloussky number satisfies the strict inequalities $2b<\sigma(S)<3b$.
The two extreme cases are as follows.
When $\sigma(S)=2b$, one obtains the \emph{Enoki surfaces}, which carry a non-empty, homologically trivial, effective divisor; in this case $S\setminus D$ is an affine bundle over an elliptic curve and admits an algebraic compactification. (Those Enoki surfaces that in addition contain an elliptic curve are known as \emph{parabolic Inoue surfaces}, or \emph{special Enoki surfaces}.)
When $\sigma(S)=3b$, one obtains the \emph{Inoue--Hirzebruch surfaces}, whose rational curves form one or two cycles.
We say that a Kato surface is \emph{hyperbolic} if it is either intermediate or of Inoue--Hirzebruch type.
In this case, there exists a negative plurisubharmonic function $G$ on the universal covering space $\tilde S \to S$, with connected fibres, called the Green function, which is pluriharmonic, surjective and submersive outside the pullback $\tilde D = p^*D$ of the rational curves, see \cite[Corollary 2.13]{DO}.
For further details on the structure and geometry of Kato surfaces we refer, for instance, to \cite{dloussky-Kato, BHPV, teleman-LNM}.

By \cite{dloussky-kohler, DO}, every Kato surface carries a global singular holomorphic foliation $\mathcal F$, smooth on the open surface $S\setminus D$.
This foliation has the following properties (see \cite[Theorem~1.9]{dloussky-kohler} for Enoki surfaces and \cite[Theorem~2.14]{DO} for hyperbolic ones): the singular
set $\mathrm{Sing}(\mathcal F)$ consists of the $b$ intersection points of the rational curves; the complement of $\mathrm{Sing}(\mathcal F)$ in each rational curve is a leaf; and every other leaf is biholomorphic to $\mathbb C$.
In the intermediate case, the proof of \cite[Theorem 5.6]{DO} shows that the pull-back foliation on the covering $\mathbb H_\ell\times\mathbb C$ of $S\setminus D$ is tangent to the fibres of the projection $\mathbb H_\ell\times\mathbb C\to\mathbb H_\ell$, where $\mathbb H_\ell=\{w\in\mathbb C:\Re w<0\}$ is the left half-plane.

We treat the Enoki and intermediate cases separately. In each case we construct smooth Hermitian metrics on the open surface $S\setminus D$ (see \Cref{prop:metriche-enoki} for the Enoki case and \Cref{prop:hyperbolic-leafwise-flatness} for the intermediate one) and study their evolution under the Chern--Ricci flow (see \Cref{thm:crf-enoki} for the Enoki case, and \Cref{thm:main-open-flow} together with \Cref{thm:green-exhaustive-GH} for the intermediate one). In the intermediate case, the canonical foliation on $S\setminus D$ has a natural flat leafwise monodromy. If $T\mathcal F$ denotes its tangent line bundle, let $\lambda\in\mathbb C^*$ be the value of the corresponding flat monodromy character on the deck transformation associated with the Green direction, and let $k=k(S)\in\mathbb Z$, $k\geq2$, be the transverse Green multiplier. We set $\mu=-2\log|\lambda|/\log k$, equivalently $k^\mu|\lambda|^2=1$. The same integer $k$ enters the topology of the complement through \cite{dloussky-oeljeklaus-toma,oeljeklaus-toma-moduli}
$$\pi_1(S\setminus D)\simeq \mathbb Z\!\left[\frac1k\right]\rtimes\mathbb Z.$$

The range $0<\mu<2$ is realised by intermediate Kato surfaces. Indeed, in the affine normal forms of \cite[Section~3.1]{oeljeklaus-toma-moduli}, the multiplier $\lambda\in\mathbb C^*$ is a complex parameter subject to $(\lambda-1)a_0=0$. In particular, when $a_0=0$, these normal-form families contain intermediate surfaces with $k^{-1}<|\lambda|<1$, which is equivalent to $0<\mu<2$.

\begin{theorem}\label{thm:main-open-kato-collapse}
Let $S$ be a Kato surface with $b_2(S)=n>0$, let $D$ be its maximal reduced divisor of rational curves, and set $A:=S\setminus D$. In the Enoki and intermediate cases below, the open surface $A$ carries smooth Hermitian metrics which are flat along the leaves of the canonical foliation. For these metrics we construct normalised Chern--Ricci solutions which collapse homothetically in the leafwise directions and converge transversely.  More precisely:
\begin{enumerate}
\item If $S$ is an Enoki surface, then, along the natural compact exhaustion of $A$, the associated metric spaces converge in the Gromov--Hausdorff sense to the elliptic curve $E$ endowed with the flat metric
$$g_{\infty,E}=-\frac{n}{2\log|\alpha|}\,g_E,$$
where $\alpha\in\mathbb C$, $0<|\alpha|<1$, is the transverse multiplier defining the elliptic quotient $E=\mathbb C^*/\langle z\mapsto\alpha z\rangle$. In particular, $\operatorname{Vol}(E,g_{\infty,E})=2\pi n=2\pi b_2(S)$.
\item Let $S$ be of intermediate type and assume
$$0<\mu:=-\frac{2\log|\lambda|}{\log k}<2,$$
where $k=k(S)\in\mathbb Z$, $k\geq2$, is the Green multiplier and $\lambda\in\mathbb C^*$ is the leafwise monodromy multiplier of the canonical foliation. Then, along the Green compact exhaustion of $A$, the associated metric spaces converge in the Gromov--Hausdorff sense to the Green circle
$$S^1_G=\left.\mathbb R\middle/\left(\frac12\log k\right)\mathbb Z\right.,$$
equipped with the flat length metric induced by $g_G=2(2-\mu)\,d\rho\otimes d\rho$.
\end{enumerate}
\end{theorem}

In the hyperbolic case, we treat the intermediate case only, since the affine Green normal form of \cite{dloussky-oeljeklaus-toma}, on which the argument rests, is stated for intermediate surfaces. We ask whether similar arguments extend to the Inoue--Hirzebruch case.

\section{Enoki surfaces}

Let $S$ be an Enoki surface, that is, a minimal surface with $b_1(S)=1$, $\mathrm{Kod}(S)=-\infty$, and $b_2(S)=n>0$, whose maximal reduced divisor of rational curves $D$ is homologically trivial and satisfies $D^2=0$. By \cite[Main Theorem]{Enoki}, $S$ is biholomorphic to the compactification of an affine bundle of degree $-n$ over an elliptic curve. More precisely, $A:=S\setminus D$ is biholomorphic to the quotient of $\mathbb C^*\times\mathbb C$ by the group generated by
$$ F(z,w) := \left( \alpha z, \; z^n w + P(z)\right), $$
where $\alpha \in \mathbb C$ satisfies $0<|\alpha|<1$, and $P(z)$ is a polynomial of degree $n$.
Throughout, we use $z$ as a coordinate on $\mathbb C^*$ and $w$ as a coordinate on $\mathbb C$.
The map $\pi \colon \mathbb C^*\times\mathbb C \to \mathbb C^*$, $\pi(z,w)=z$, induces the affine bundle $\pi \colon A \to E$ over the elliptic curve $E:=\mathbb C^*/\langle z \mapsto \alpha z \rangle$.

\subsection{Adapted coframe}

The holomorphic $1$-form $\zeta:=dz/z$ on $\mathbb C^*\times\mathbb C$ is $F$-invariant and hence descends to $A$. Its kernel defines the holomorphic foliation $\mathcal F\vert_A$, whose leaves are the fibres of $\pi\colon A \to E$. This foliation can be extended to a singular holomorphic foliation on $S$, see \cite{DO}.

We fix a smooth function $s\colon \mathbb C^*\to\mathbb C$ such that $s(\alpha z)=z^ns(z)-P(z)$.
To construct such a function, take a smooth section $\xi \colon E \to A$, obtained by gluing local sections together via a partition of unity. A lift of $\xi$ to $\mathbb C^*\times\mathbb C$ induces a smooth equivariant function $\tilde s\colon \mathbb C^* \to \mathbb C$. Set $s:=-\tilde s$.
The action of $F$ is linearised by setting
\begin{equation}\label{eq:tilde-w}
\tilde w:=w+(s\circ\pi),
\end{equation}
so that $\tilde w \circ F = z^n\tilde w$.

On $\mathbb{C}^{*}\times\mathbb{C}$ we consider the coframe of $(1,0)$-forms
$$ \zeta:=\frac{dz}{z}, \qquad \eta:=dw+B(z,w)\,dz, $$
where
$$ B(z,w):=\frac{\partial s}{\partial z}+\frac{q'(\log|z|)}{2z}\,\tilde{w}, \qquad
q(\sigma):=n\sigma-\frac{n}{\log|\alpha|}\,\sigma^{2}. $$
As noted above, $F^{*}\zeta=\zeta$, and a direct computation gives $F^{*}\eta=z^{n}\eta$. Thus $\zeta$ descends to $A$, whereas $\eta$ is an equivariant vertical coform on the covering; together they form a coframe adapted to the foliation.

This coframe has a natural interpretation in terms of the line bundle $L\to E$ given by the linear part of the affine bundle $A\to E$. Equip $L$ with the Hermitian fibre metric whose weight in this trivialisation is $\Lambda:=e^{q(\log|z|)}$, where, here and below, $q'$ and $q''$ are evaluated at $\log|z|$. Since $\partial\log|z|=\tfrac12\zeta$, we have $\partial\log\Lambda=\tfrac{q'}{2}\,\zeta$, equivalently $\partial\Lambda=\tfrac{q'}{2}\,\Lambda\,\zeta$, so that the Chern connection of $(L,\Lambda)$ has connection form $\tfrac{q'}{2}\,\zeta$. Its $(1,0)$-part consequently acts on the linearising
coordinate $\tilde w$ by
$$ \nabla^{1,0}\tilde{w} = \partial\tilde{w}+\frac{q'}{2}\,\tilde{w}\,\zeta = dw+\left(\frac{\partial s}{\partial z}+\frac{q'(\log|z|)}{2z}\,\tilde{w}\right)dz = \eta. $$

The structure equations are
\begin{align*}
d\zeta &= 0, \\
\partial\eta &= \frac{1}{2}q'(\log|z|)\,\eta\wedge\zeta, \\
\bar\partial\eta &= -\left(\frac{1}{4}q''(\log|z|)\,\tilde{w}
+|z|^{2}\frac{\partial^{2}s}{\partial z\,\partial\bar z}
+\frac{1}{2}q'(\log|z|)\,\bar z\,\frac{\partial s}{\partial\bar z}\right)\zeta\wedge\bar\zeta.
\end{align*}

\subsection{Special Hermitian metrics on the open surface}

A direct computation gives $q(\log|\alpha z|) = q(\log|z|) - 2n\log|z|$. Hence, for $\Lambda:=e^{q(\log|z|)}$ as before, $F^{*}\Lambda=|z|^{-2n}\Lambda$.
Hence the $(1,1)$-forms
\begin{equation}\label{eq:metriche-enoki}
\omega_{f,g}:=\sqrt{-1}\,f\,\zeta\wedge\bar\zeta
+\sqrt{-1}\,g\,\Lambda\,\eta\wedge\bar\eta,
\end{equation}
with $f,g\in\mathcal{C}^{\infty}(A,\mathbb{R})$, are well-defined on $A$; if $f,g>0$, they are Hermitian metrics. In particular, for $f>0$ and a constant $c>0$, the metric $\omega_{f,c}$ is flat along the leaves. Indeed, after lifting a leaf $L=\pi^{-1}([z_0])$ to $\{z=z_0\}$, one has
$$\omega_{f,c}\big|_{L}=\sqrt{-1}\,c\,\Lambda(z_0)\,dw\wedge d\bar w.$$

For positive $f$ and $g$, the Chern--Ricci form is
\begin{align*}
\operatorname{Ric}(\omega_{f,g})
&= -\sqrt{-1}\partial\overline\partial\log\det\omega_{f,g} \\
&= -\sqrt{-1}\partial\overline\partial\log\left(\frac{\Lambda}{|z|^2}fg\right) \\
&= -\sqrt{-1}\partial\overline\partial\left(q+\log f+\log g\right) \\
&= \frac{n}{2\log|\alpha|}\sqrt{-1}\zeta\wedge\bar\zeta -\sqrt{-1}\partial\overline\partial \log f -\sqrt{-1}\partial\overline\partial \log g.
\end{align*}
Set $\omega_E:=\sqrt{-1}\,\zeta\wedge\bar\zeta$ and let $g_E$ be its associated Riemannian metric. With the non-negative Laplacian $\Delta_E=d^*d$, our convention gives
$$-\sqrt{-1}\,\partial\overline\partial u=\frac12(\Delta_Eu)\,\omega_E$$
for every smooth function $u$ on $E$. Hence, if $f=F\circ\pi$ for a positive function $F\in\mathcal C^\infty(E,\mathbb R)$ and $g=c>0$ is constant, then
\begin{equation}\label{eq:ricci-1}
\operatorname{Ric}(\omega_{F\circ \pi,c}) = \sqrt{-1} \left(\frac{n}{2\log|\alpha|}+\frac{1}{2}(\Delta_E\log F) \circ \pi\right)\,\zeta\wedge\bar\zeta.
\end{equation}
When $f=a>0$ is constant,
$$ \operatorname{Ric}(\omega_{a,c})=\frac{n}{2\log|\alpha|}\sqrt{-1}\zeta\wedge\bar\zeta \leq 0 $$
is negative semidefinite (recall that $0<|\alpha|<1$), with kernel defining the foliation.

We also compute $\partial\overline\partial(\Lambda\,\eta\wedge\bar\eta)=-\frac{1}{4}q''\Lambda\,\zeta\wedge\bar\zeta\wedge\eta\wedge\bar\eta$. If, in addition, $g$ is the pull-back of a smooth function on $E$, denoted by the same letter, then
$$ \sqrt{-1}\partial\overline\partial\omega_{f,g} = - \left(\frac{\partial^2 f}{\partial w \partial \bar w}+\frac{1}{4}\Lambda\left(-2\Delta_E g-q''g\right)\right)\,\zeta\wedge\bar\zeta\wedge\eta\wedge\bar\eta. $$
For $f=a-\frac{n}{2\log|\alpha|}c\Lambda|\tilde w|^2$ and $g=c$, with $a,c>0$, one obtains
$$ \sqrt{-1}\partial\overline\partial\omega_{a-\frac{n}{2\log|\alpha|}c\Lambda|\tilde w|^2,c} = 0, $$
so the Gauduchon condition is satisfied.
When $f$ is constant or depends only on the base $E$, as $f=F \circ \pi$ with $F\in\mathcal C^\infty(E,\mathbb R)$, and $g=c>0$,
$$ \sqrt{-1}\partial\overline\partial\omega_{F\circ\pi,c} = \Lambda \frac{n}{2\log|\alpha|} c \, \left(\sqrt{-1}\,\zeta\wedge\bar\zeta\right) \wedge \left(\sqrt{-1}\,\eta\wedge\bar\eta\right) <0, $$
so the metric is plurisigned negative in the sense of \cite{angella-guedj-lu}.
The metric $\omega_{f,c}$ is also complete whenever $\inf_A f>0$ and $c>0$. Since $\omega_{f,c} \geq \min\{\inf f, c\}\,\omega_{1,1}$, it suffices to prove that $\omega_{1,1}$ is complete. The function $h:=\Lambda|\tilde w|^2$ is the squared fibre norm over the compact base $E$, hence is proper; the connection formula for $\eta$ gives $|dh|_{\omega_{1,1}}\leq C\sqrt h$. Thus $u:=\sqrt{1+h}$ is a proper Lipschitz function with respect to $d_{\omega_{1,1}}$. Integration along curves then shows that every closed metric ball is contained in a compact set.

We summarise the preceding computations as follows.

\begin{propo}\label{prop:metriche-enoki}
Let $S$ be an Enoki surface and let $D$ be its maximal reduced divisor of rational curves. Then the forms $\omega_{f,c}$ of \eqref{eq:metriche-enoki}, with $f$ smooth and positive on $S\setminus D$ and $c>0$, define smooth Hermitian metrics on $S\setminus D$ which are flat along the leaves of the canonical foliation and complete whenever $\inf_A f>0$.
For the choice $f=a-\frac{n}{2\log|\alpha|}c\Lambda|\tilde w|^2$, with $a>0$, the metric $\omega_{f,c}$ is in addition Gauduchon.
\end{propo}

\begin{remark}
It is natural to ask whether the metrics $\omega_{f,g}$ of \Cref{prop:metriche-enoki} extend to possibly singular Hermitian metrics on the whole surface $S$.
\end{remark}

\subsection{Explicit solutions to the Chern--Ricci flow}

On the Enoki surface $S$, endowed with a given Hermitian metric $\omega_0$, we consider the Chern--Ricci flow
$$ \frac{\partial\omega(t)}{\partial t}=-\operatorname{Ric}(\omega(t)), \qquad \omega(0)=\omega_0. $$
There exists a unique solution on some maximal interval $[0,T)$ with $T>0$.

Suppose now that $\omega_0$ is Gauduchon. Then $\omega(t)$ remains Gauduchon for as long as the flow exists. By \cite[Corollary 1.4]{tosatti-weinkove} and the discussion in the introduction, its maximal existence time is finite. By contrast, for the special metrics constructed above the flow on $S\setminus D$ exists for all $t\geq0$ and retains the topological information described below.

\begin{propo}\label{prop:crf}
Let $S$ be an Enoki surface and let $D$ be its maximal reduced divisor of rational curves. On $A:=S\setminus D$, consider the smooth Hermitian metric
$$ \omega_0 = \sqrt{-1}\, (F_0\circ\pi)\, \zeta\wedge\bar\zeta + \sqrt{-1}\, c_0\, \Lambda\, \eta\wedge\bar\eta, $$
where $F_0$ is a smooth positive function on the elliptic curve $E$ and $c_0>0$.
Then there exists a smooth solution $\omega(t)$ of the Chern--Ricci flow, defined for all $t \geq 0$, with $\omega(0)=\omega_0$.
\end{propo}

\begin{proof}
Consider the family of Hermitian metrics
$$ \omega(t) = \sqrt{-1}\, (F(t)\circ\pi)\, \zeta\wedge\bar\zeta + \sqrt{-1}\, c_0\, \Lambda\, \eta\wedge\bar\eta, $$
where $F(t)$ is a smooth function on $E$ with $F(0)=F_0$.
By \eqref{eq:ricci-1},
\begin{align*}
\operatorname{Ric}(\omega(t))
&=\sqrt{-1}\left(\frac{n}{2\log|\alpha|}+\frac12(\Delta_E\log F(t))\circ\pi\right)\zeta\wedge\bar\zeta.
\end{align*}
Thus $\omega(t)$ solves the Chern--Ricci flow starting at $\omega_0$ if and only if $F(t)$ solves the following scalar parabolic equation on $E$:
\begin{equation}\label{eq:par-scalar}
\frac{\partial F(t)}{\partial t} = - \frac{1}{2}\Delta_E \log F(t) - \frac{n}{2\log|\alpha|}, \qquad F(0)=F_0.
\end{equation}
For the geometric Laplacian $\Delta_E=d^{*}d$ on $E$, with non-negative eigenvalues, one has
$$\Delta_E\log F=\frac{\Delta_EF}{F}+\frac{|\nabla_EF|^2}{F^2}.$$
Hence \eqref{eq:par-scalar} can be rewritten as
$$ \frac{\partial F(t)}{\partial t}
= -\frac{\Delta_E F(t)}{2F(t)}
- \frac{|\nabla_E F(t)|^2}{2(F(t))^2}
- \frac{n}{2\log|\alpha|}, \qquad F(0)=F_0. $$
Since $\Delta_E$ has non-negative spectrum, the leading operator $-\tfrac{1}{2F}\Delta_E$ is the (forward) heat operator; as $F_0 \geq \min_E F_0>0$, the quasilinear equation is uniformly parabolic on a short time interval $[0,\varepsilon)$.

Standard quasilinear parabolic theory (see, for instance, \cite{lieberman, LSU}) gives $T_0>0$ and a unique smooth solution $F(t)$ on $[0,T_0]$. We prove that the maximal existence time is $T_{\rm max}=+\infty$.

The required bounds follow from the maximum principle. For $T>0$ such that \eqref{eq:par-scalar} has a positive solution $F(t)$, let $\phi,\psi \in \mathcal C^1([0,T])$ be such that
$$ \psi'(t) \leq -\frac{n}{2\log|\alpha|} \leq \phi'(t), $$
$$ \psi(0) \leq \min_E F(0) \leq F(0) \leq \max_E F(0) \leq \phi(0). $$
Then, for every $t \in [0,T]$,
$$ \psi(t) \leq F(t) \leq \phi(t) $$
on $E$. Indeed, fix $\varepsilon>0$ and set $w:=F-\phi-\varepsilon t$. At $t=0$, we have $w(0)\leq 0$. Let $(x_0,t_0)\in E \times [0,T]$ be a maximum point for $w$.
At a spatial maximum $x_0$ of $F(\cdot,t_0)$ one has $\nabla_E F=0$ and $\Delta_E F\ge0$, hence
$(\Delta_E\log F)(x_0,t_0)=\frac{(\Delta_E F)(x_0,t_0)}{F(x_0,t_0)}\ge0$.
We may assume $t_0>0$. Since $(x_0,t_0)$ is a maximum point with $t_0>0$, one has $\partial_t w(x_0,t_0)\geq0$. On the other hand,
\begin{align*}
\frac{\partial w}{\partial t}(x_0,t_0)
=& \frac{\partial F}{\partial t}(x_0,t_0) - \phi'(t_0) - \varepsilon \\
=& - \frac{1}{2} (\Delta_E\log F)(x_0,t_0) - \frac{n}{2\log|\alpha|} - \phi'(t_0) - \varepsilon\\
\leq& - \frac{n}{2\log|\alpha|} - \phi'(t_0) - \varepsilon \leq -\varepsilon<0.
\end{align*}
The lower bound follows in the same way by applying the argument to $v:=F-\psi+\varepsilon t$.

Applying these barriers with
$$ \phi(t) := \max_E F_0 - \frac{n}{2\log|\alpha|} t, \qquad \psi(t) := \min_E F_0 - \frac{n}{2\log|\alpha|} t, $$
and letting $T\uparrow T_{\rm max}$ yields
$$ \min_E F_0 - \frac{n}{2\log|\alpha|} t \leq F(t) \leq \max_E F_0 - \frac{n}{2\log|\alpha|}t, $$
on $E$ for every $t\in[0,T_{\rm max})$.

Suppose, for contradiction, that $T_{\rm max}<\infty$. Since $0<|\alpha|<1$, the barrier estimate gives
$$0<\min_E F_0\leq F(t)\leq \max_E F_0-\frac{n}{2\log|\alpha|}\,T_{\rm max}<+\infty$$
for $0\leq t<T_{\rm max}$. Thus \eqref{eq:par-scalar} remains uniformly parabolic up to $T_{\rm max}$. Equivalently, setting $u=\log F$,
$$\frac{\partial u}{\partial t}=-\frac12e^{-u}\Delta_Eu-\frac{n}{2\log|\alpha|}e^{-u},$$
and the coefficient $e^{-u}/2$ is uniformly bounded above and away from zero.
In local coordinates, the equation for $u$ is a uniformly parabolic equation in non-divergence form whose leading coefficient is
$$
a(x,t)=\frac12e^{-u(x,t)}.
$$
The preceding $\mathcal C^0$ estimate implies that $a$ is uniformly bounded above and away from zero, while the zeroth-order term is uniformly bounded. The Krylov--Safonov estimate therefore gives a uniform $\mathcal C^\alpha$ bound for $u$ up to $T_{\rm max}$; see, for instance, \cite[Chapter~IV]{lieberman}. Since $a=\frac12e^{-u}$ is then uniformly $\mathcal C^\alpha$, the parabolic Schauder estimates and bootstrapping yield uniform $\mathcal C^k$ bounds for $u$, and hence for $F$, for every $k$; see \cite[Chapter~V]{lieberman}. The standard continuation criterion for quasilinear uniformly parabolic equations then extends the solution beyond $T_{\rm max}$, contradicting maximality. Hence $T_{\rm max}=+\infty$.
\end{proof}

\begin{remark}
The metrics $\omega(t)$ in \Cref{prop:crf} remain flat along the leaves, and $\sqrt{-1}\partial\overline\partial\omega(t)=\sqrt{-1}\partial\overline\partial\omega_0$ remains a constant negative multiple of $\sqrt{-1}\,\zeta\wedge\bar\zeta\wedge\sqrt{-1}\,\Lambda\,\eta\wedge\bar\eta$.
\end{remark}

\begin{remark}
For $\omega_0=\sqrt{-1}\,a_0\,\zeta\wedge\bar\zeta+\sqrt{-1}\,c_0\Lambda\,\eta\wedge\bar\eta$, the explicit solution is
$$ \omega(t) = \omega_0 - \frac{n}{2\log|\alpha|}\,t\, \sqrt{-1}\,\zeta\wedge\bar\zeta, $$
immortal but not eternal, with $\operatorname{Ric}(\omega(t))=\frac{n}{2\log|\alpha|} \, \sqrt{-1}\,\zeta\wedge\bar\zeta$ constant.
\end{remark}

\subsection{Gromov--Hausdorff convergence}

\begin{theorem}\label{thm:crf-enoki}
Let $S$ be an Enoki surface and let $D$ be its maximal reduced divisor of rational curves. On $A:=S\setminus D$, consider the smooth Hermitian metric
$$ \omega_0 = \sqrt{-1}\,(F_0\circ\pi)\,\zeta\wedge\bar\zeta+\sqrt{-1}\,c_0\,\Lambda\,\eta\wedge\bar\eta $$
with $F_0$ a smooth positive function on $E$ and $c_0>0$. Then there exists a smooth solution $\tilde\omega(t)$, defined for all $t\geq0$, of the normalised Chern--Ricci equation
$$
\frac{\partial\tilde\omega(t)}{\partial t}
=
-\operatorname{Ric}(\tilde\omega(t))-\tilde\omega(t),
\qquad
\tilde\omega(0)=\omega_0.
$$
Moreover,
$$ \tilde\omega(t) \to -\frac{n}{2\log|\alpha|}\,\sqrt{-1}\,\zeta\wedge\bar\zeta $$
as $t\to+\infty$, uniformly on $A$ with respect to the reference metric $\hat\omega:=\sqrt{-1}\,\zeta\wedge\bar\zeta+\sqrt{-1}\,\Lambda\,\eta\wedge\bar\eta$. For every $R>0$, the compact sublevel $K_R:=\{\Lambda\,|\tilde w|^2 \leq R\}$, where $\tilde w$ is defined in \eqref{eq:tilde-w}, satisfies
$$ (K_R, d_t) \stackrel{\rm GH}{\to} \left(E, \sqrt{-\frac{n}{2\log|\alpha|}}\, d_{g_E}\right), $$
where $d_t$ denotes the ambient distance on $A$ induced by $\tilde\omega(t)$, restricted to $K_R\times K_R$. The limiting metric space is associated with the flat metric $-\frac{n}{2\log|\alpha|} \,g_E$, where $g_E$ is the Riemannian metric associated with $\sqrt{-1}\,\zeta\wedge\bar\zeta$.
\end{theorem}

\begin{proof}
By \Cref{prop:crf}, the Chern--Ricci flow starting at $\omega_0$ admits a solution $\omega(s)$ defined for every $s\geq 0$. The relation
$$ \tilde\omega(t):=e^{-t}\omega(e^t-1) $$
gives a solution of the \emph{normalised} Chern--Ricci flow for every $t\geq 0$.
The solution has the form
$$ \tilde\omega(t) = \sqrt{-1}\,( F(t) \circ \pi) \, \zeta\wedge\bar\zeta + \sqrt{-1}\, c_0e^{-t}\, \Lambda\, \eta\wedge\bar\eta, $$
where $F(t)$ is a smooth positive function on $E$ satisfying
$$ \frac{\partial F}{\partial t} = - \frac{1}{2} \Delta_E \log F - F - \frac{n}{2\log|\alpha|}. $$
Set $c:=-n/(2\log|\alpha|)>0$. The normalised equation becomes
$$\frac{\partial F}{\partial t}=-\frac12\Delta_E\log F-F+c.$$
The functions
$$F_-(t):=c+e^{-t}(\min_E F_0-c),\qquad F_+(t):=c+e^{-t}(\max_E F_0-c)$$
are, respectively, a subsolution and a supersolution.
They in fact solve the normalised equation exactly.
By the same maximum principle argument as in Proposition \ref{prop:crf}, applied to the normalised equation, one obtains $F_-(t)\leq F(t)\leq F_+(t)$ and
$$\|F(t)-c\|_{\mathcal C^0(E)}\leq e^{-t}\|F_0-c\|_{\mathcal C^0(E)}.$$
Set $F_-:=\min\{\min_E F_0,c\}$ and $F_+:=\max\{\max_E F_0,c\}$. Then $0<F_-\leq F(t)\leq F_+$ for all $t\geq0$. Thus $F$ remains uniformly bounded above and away from zero. Applied to $F-c$, standard parabolic estimates on time-translated unit intervals,  interpolating with the preceding $\mathcal C^0$ convergence and bootstrapping, give $F(t)\to c$ in $\mathcal C^\infty(E)$. The $\mathcal C^0$ estimate also supplies the exponential rate required for the metric estimate:
\begin{align*}
\left| \tilde\omega(t) - \sqrt{-1}\, \left(-\frac{n}{2\log|\alpha|}\right)\, \zeta\wedge\bar\zeta \right|_{\hat\omega}
&\leq \left| F(t) - \left(-\frac{n}{2\log|\alpha|}\right) \right| + c_0 e^{-t} \\
&\leq (\|F_0-c\|_{\mathcal C^0(E)}+c_0)e^{-t}
\end{align*}
on $A$. Thus $\tilde\omega(t) \to -\sqrt{-1}\,\frac{n}{2\log|\alpha|}\, \zeta\wedge\bar\zeta$ uniformly on $A$ with respect to $\hat\omega$.

It remains to prove the Gromov--Hausdorff convergence of $K_R$. Write $g_t:=g_{\tilde\omega(t)}$ for the associated Riemannian metric. Since $\Psi:=\Lambda|\tilde w|^2$ is a proper exhaustion of $A$, the sublevel $K_R$ is compact and $\pi(K_R)=E$.
We show that $\pi\colon(K_R,d_t)\to(E,h_t)$, with $h_t:=F(t)\,g_E$, is an $\varepsilon_t$-Gromov--Hausdorff
approximation with $\varepsilon_t\to0$. Together with $h_t\to (-\frac{n}{2\log|\alpha|})\,g_E$, this will prove the claim.

In the coframe $(\zeta,\eta)$, the $\zeta$-directions and the $\eta$-directions are orthogonal, and $\pi$ is a Riemannian submersion. Hence $\pi$ is $1$-Lipschitz: for $p,q\in K_R$ one has
$$ d_{h_t}(\pi(p),\pi(q)) \leq d_t(p,q). $$

On a fibre $\pi^{-1}(z_0)$, one has $z=z_0$ constant, hence $\eta\vert_{\pi^{-1}(z_0)}=dw=d\tilde w$ and the induced Riemannian metric is $2c_0e^{-t}\Lambda(z_0)\,|d\tilde w|^2$.
We have
\begin{align*}
\operatorname{diam}_{g_t} \left(K_R \cap \pi^{-1}(z_0)\right)
&\leq 2 \sqrt{\frac{R}{\Lambda(z_0)}} \cdot \sqrt{2c_0 e^{-t} \Lambda(z_0)} \\
&= 2\sqrt{2c_0R}\, e^{-t/2} =: \varepsilon_R(t),
\end{align*}
uniformly in $z_0$, and $\varepsilon_R(t)\to0$ as $t\to+\infty$.

Set $X:=z\partial_z-zB\,\partial_w$, the horizontal $(1,0)$-vector satisfying $\zeta(X)=1$ and $\eta(X)=0$. Recall that $B=\frac{\partial s}{\partial z}+\frac{q'}{2z}\tilde w$. In every local trivialisation of the line bundle $L\to E$, the real horizontal lift equation for the fibre coordinate has the form
$$\frac{d\tilde w}{d\sigma}=a(\sigma)\tilde w+b(\sigma),$$
where the coefficients are smooth functions of the base variables. Fix a finite trivialising cover of the compact curve $E$, together with a Lebesgue number for a subordinate finite refinement. The transition functions and the corresponding coefficients are uniformly bounded on this cover. Moreover, the $g_E$-length of every minimising $h_t$-geodesic is bounded by
$$L_0:=\sqrt{\frac{F_+}{F_-}}\,\operatorname{diam}(E,g_E),$$
uniformly in $t$. Subdividing such a geodesic into a uniformly bounded number of subarcs lying in the chosen trivialising charts and applying Gr\"onwall successively gives a constant $C'>0$, independent of $t$ and of the geodesic, such that $|\tilde w|\leq C'$ along every horizontal lift starting on the chosen smooth section $\{\tilde w=0\}$. Consequently, since $\Psi=\Lambda|\tilde w|^2$ is globally defined and proper, there exists $R'>R$, independent of $t$, such that every such lift remains in $K_{R'}$.

Take $p,q\in K_R$. Join $p$ to the point $p_0$ of the section $\{\tilde w=0\}$ over $\pi(p)$ by a curve of length at most $\varepsilon_R(t)$. Let $\gamma$ be a minimising $h_t$-geodesic from $\pi(p)$ to $\pi(q)$. Lift $\gamma$ horizontally from $p_0$ and parametrise the lift by $g_E$-arclength $\sigma\in[0,\ell]$, where
$$\ell\leq\frac{1}{\sqrt{F_-}}\,\operatorname{diam}(E,h_t)\leq\sqrt{\frac{F_+}{F_-}}\,\operatorname{diam}(E,g_E),$$
uniformly in $t$.
By the preceding estimate, the lift remains in $K_{R'}$ and ends at $q_0$ over $\pi(q)$. Its length is at most $d_{h_t}(\pi(p),\pi(q))$. Finally, join $q_0$ to $q$ by a curve of length at most $\varepsilon_{R'}(t)$.
Combining these estimates gives
$$ d_t(p,q) \leq d_{h_t}(\pi(p),\pi(q)) + \varepsilon_R(t) + \varepsilon_{R'}(t), $$
where $\varepsilon_{R}(t)\to 0$ and $\varepsilon_{R'}(t)\to 0$ as $t\to+\infty$.
Since $h_t=F(t)g_E\to c\,g_E$ uniformly, set
$$\delta_t:=\sup_{x,y\in E}\left|d_{h_t}(x,y)-\sqrt c\,d_{g_E}(x,y)\right|,$$
so that $\delta_t\to0$. Combining the preceding upper estimate with the $1$-Lipschitz lower estimate gives
$$\sup_{p,q\in K_R}\left|d_t(p,q)-\sqrt c\,d_{g_E}(\pi(p),\pi(q))\right|\leq\varepsilon_R(t)+\varepsilon_{R'}(t)+\delta_t\longrightarrow0.$$
Since $\pi(K_R)=E$, the map $\pi:K_R\to E$ is a Gromov--Hausdorff approximation whose distortion tends to zero. This proves the claim.
\end{proof}

\begin{remark}
The limiting metric also satisfies
\begin{align*}
\operatorname{Vol}\left(E, -\frac{n}{2\log|\alpha|} \,g_E\right)
&= \left(-\frac{n}{2\log|\alpha|}\right) \cdot \left(4\pi(-\log|\alpha|)\right) \\
&= 2\pi n = 2\pi b_2(S),
\end{align*}
which reflects the topological information carried by $b_2(S)$.
\end{remark}

\begin{remark}
The restriction to the compact exhaustion $\{K_R\}_{R>0}$ is essential. For every finite $t$, the metric $\tilde\omega(t)$ is complete: its horizontal coefficient is bounded below by $F_->0$, while its vertical coefficient is $c_0e^{-t}>0$, so completeness follows from \Cref{prop:metriche-enoki}. Since $A$ is non-compact, the complete Riemannian manifold $(A,d_t)$ has infinite diameter. It therefore cannot converge globally, in the compact Gromov--Hausdorff topology, to the compact elliptic base.
\end{remark}

\section{Kato surfaces of intermediate type}

Let $S$ be a Kato surface of intermediate type and let $D$ be its maximal reduced divisor of rational curves. For surfaces of index one, Dloussky--Oeljeklaus--Toma \cite{dloussky-oeljeklaus-toma} identify the universal cover of $S\setminus D$ with $X=\mathbb H_\ell\times\mathbb C$, where $\mathbb H_\ell=\{\Re w<0\}$, and put the deck action into an affine triangular normal form. The higher-index case is obtained from this model by a finite cyclic extension of the deck group. Set $r:=\operatorname{ind}(S)$. We recall the form needed below.

 \begin{propo}\label{prop:higher-index-green-model}
There exists an intermediate surface $S_1$ of index one, with maximal reduced divisor of rational curves $D_1$, and an unramified cyclic covering
$$
A_1:=S_1\setminus D_1\longrightarrow A:=S\setminus D
$$
of degree $r$. On their common universal cover $X=\mathbb H_\ell\times\mathbb C$ one may choose coordinates so that the deck group $\Gamma_1$ of $A_1$ is generated by
$$
T(w,z)=(w+2\pi\sqrt{-1},z),\qquad
G(w,z)=(kw,\lambda z+h(w)),
$$
where $k=k(S)\in\mathbb Z$, $k\geq2$, and the full deck group $\Gamma$ of $A$ is generated by $\Gamma_1$ and
$$
\varphi(w,z)=\left(w-\frac{2\pi\sqrt{-1}}{r},\xi z\right),
\qquad
\xi:=e^{2\pi\sqrt{-1}\sigma_0/r},
$$
where $\sigma_0$ is the largest exponent occurring in the polynomial $Q$ of the associated index-one normal form; see \cite[Sections~3.1--3.2]{oeljeklaus-toma-moduli}. In particular, $\xi^r=1$ and $|\xi|=1$.
Here $r$ divides $k-1$, and
$$
\varphi^r=T^{-1},\ 
\varphi T\varphi^{-1}=T,\ 
\varphi G\varphi^{-1}=T^{(k-1)/r}G.
$$
In particular $G\varphi G^{-1}=\varphi^k$, so that
$\Gamma\simeq\mathbb Z[1/k]\rtimes\mathbb Z$. Moreover, $\varphi$ preserves
$s:=-\Re w$.
\end{propo}

\begin{proof}
For $r=1$ this is the affine Green normal form of
\cite{dloussky-oeljeklaus-toma}. For $r>1$, Oeljeklaus--Toma
\cite[Section~3.2 and Remark~4.4]{oeljeklaus-toma-moduli} show that every
intermediate surface of index $r$ is obtained as a cyclic quotient of an
index-one model. In their coordinates $(z,w)\in\mathbb C\times\mathbb H_\ell$,
the additional normalising transformation has the form
$$
(z,w)\longmapsto
\left(e^{2\pi\sqrt{-1}\sigma_0/r}z,
w-\frac{2\pi\sqrt{-1}}{r}\right).
$$
Interchanging the two coordinates gives the form above. 
The normal form satisfies the property $(I_q)$ of
\cite[Section~3.2, p.~329]{oeljeklaus-toma-moduli}, with $q=r$ in our notation. Thus, writing
$$
h(w)=Q(e^{-w}),\qquad Q(\zeta)=\sum_{m=\ell}^{\sigma} b_m\zeta^m,
$$
one has $a_0=0$ and
$
r\mid(k-1),\ r\mid(m-m')
$
whenever $b_m b_{m'}\neq0$. Since $\sigma_0$ is an exponent with $b_{\sigma_0}\neq0$, every exponent $m$ with $b_m\neq0$ satisfies
$
m\equiv\sigma_0\pmod r.
$
Therefore, with $\xi=e^{2\pi\sqrt{-1}\sigma_0/r}$,
$$
h\left(w-\frac{2\pi\sqrt{-1}}{r}\right)
=
\sum_m b_m e^{-mw}e^{2\pi\sqrt{-1}m/r}
=
\xi\,h(w).
$$
It follows that
$
\varphi G\varphi^{-1}=T^{(k-1)/r}G.
$
Since $\xi^r=1$, one also has $\varphi^r=T^{-1}$. Therefore
$$
\varphi G\varphi^{-1}=\varphi^{-(k-1)}G,
$$
equivalently
$$
G\varphi G^{-1}=\varphi^k.
$$
The resulting presentation is the standard presentation of
$\mathbb Z[1/k]\rtimes\mathbb Z$. Finally, $\varphi^*s=s$ is immediate.
\end{proof}

We work henceforth on this common universal cover and with the full deck group $\Gamma$. The construction below produces a $\Gamma$-equivariant Hermitian metric and an explicit solution of the normalised Chern--Ricci flow on $S\setminus D$.

\subsection{The affine Green model}
Each element $\gamma\in\Gamma$ has the affine triangular form
\begin{equation}\label{eq:general-affine-element}
 \gamma(w,z)=\bigl(a_\gamma w+b_\gamma,\lambda_\gamma z+h_\gamma(w)\bigr),
\end{equation}
where $a_\gamma>0$, $\lambda_\gamma\in\mathbb{C}^*$, and $h_\gamma \colon \mathbb H_\ell \to \mathbb C$ is holomorphic.
For the distinguished elements of the index-one subgroup we write
$$ T(w,z)=(w+2\pi\sqrt{-1},z),
\quad \text{ and } \quad
G(w,z)=(kw,\lambda z+h(w)), $$
where $k=k(S)\in\mathbb Z$, $k\geq2$. If $r>1$, the full group also contains the transformation $\varphi$ of \Cref{prop:higher-index-green-model}.
Set $s:=-\Re w>0$ and $\theta:=\Im w$, so that $w=-s+\sqrt{-1}\theta$.
Then
\begin{equation}\label{eq:s-transforms}
 T^*s=s,
 \qquad
 G^*s=ks.
\end{equation}
Define the transverse Green form by
\begin{equation*}
 \omega_{\rm tr}:=\ddbar(-\log s).
\end{equation*}
Since $s=-(w+\bar w)/2$, we have $\partial s=-\tfrac12\,dw$ and $\overline\partial s=-\tfrac12\,d\bar w$. Hence
$$ \omega_{\rm tr} = \sqrt{-1} \partial\overline\partial (-\log s) =
 \frac{\sqrt{-1}}{4s^2}\,dw\wedge d\bar w. $$
The transformation law of $s$ gives
\begin{equation*}
 T^*\omega_{\rm tr}=\omega_{\rm tr},
 \quad
 G^*\omega_{\rm tr}
 =
 \ddbar(-\log(ks))
 =
 \ddbar(-\log s)
 =
 \omega_{\rm tr}.
\end{equation*}
For $r>1$, one also has $\varphi^*s=s$ and hence $\varphi^*\omega_{\rm tr}=\omega_{\rm tr}$. Thus $\omega_{\rm tr}$ is invariant under the full deck group and descends to $S\setminus D$.

\subsection{The cut-off and the equivariant vertical coform}
The vertical tangent distribution is $V:=\mathbb{C}\cdot\partial_z\subset T_X^{1,0}$.
It is preserved by $\Gamma$: for an element
$\gamma\in\Gamma$ as in \eqref{eq:general-affine-element}, one has
\begin{equation*}
 d\gamma(\partial_z)=\lambda_\gamma\partial_z.
\end{equation*}
In particular, $dT(\partial_z)=\partial_z$ and $dG(\partial_z)=\lambda\partial_z$; when $r>1$, $d\varphi(\partial_z)=\xi\partial_z$.

Since the action of $\Gamma$ on $X$ is free and properly discontinuous, the
quotient map $q:X\to X/\Gamma=A=S\setminus D$ is a smooth covering. We fix once and for all a standard cut-off function for this covering: a smooth function
$\chi\in\mathcal C^\infty(X)$, $\chi\geq0$, such that the family
$\{\chi\circ\gamma\}_{\gamma\in\Gamma}$ is locally finite and
\begin{equation}\label{eq:cutoff-sum}
 \sum_{\gamma\in\Gamma} \chi(\gamma p)=1
\end{equation}
for every $p\in X$. Such a function is obtained in the usual way by lifting a locally
finite partition of unity on the quotient to chosen sheets over evenly covered
coordinate neighbourhoods and extending by zero; proper discontinuity gives
the required local finiteness.

Starting with the Euclidean Hermitian metric $g_0=|dw|^2+|dz|^2$, we set
\begin{equation*}
 g:=\sum_{\gamma\in\Gamma}\chi(\gamma p)\,\gamma^*g_0.
\end{equation*}
The sum is locally finite, hence $g$ is a smooth Hermitian metric. It is positive because the coefficients $\chi(\gamma p)$ are non-negative, at least
one of them is positive at each point, and every $\gamma^*g_0$ is positive.
It is also $\Gamma$-invariant: for $\eta\in\Gamma$,
\begin{align*}
 (\eta^*g)_p
 &=
 \sum_{\gamma\in\Gamma}\chi(\gamma\eta p)(\gamma\eta)^*g_0|_p \\
 &=
 \sum_{\delta\in\Gamma}\chi(\delta p)\delta^*g_0|_p
 =
 g_p,
\end{align*}
where $\delta=\gamma\eta$.

Let $H\subset T_X^{1,0}$ be the $g$-orthogonal complement of $V$. Since
both $g$ and $V$ are $\Gamma$-invariant, so is $H$. There is a unique
smooth $(1,0)$-form $\Theta$ such that $\Theta(\partial_z)=1$ and $\Theta|_H=0$. Write
\begin{equation*}
 \Theta=dz+B\,dw.
\end{equation*}
In local coordinates, write
\begin{equation*}
 g
 =
 g_{w\bar w}\,dw\,d\bar w
 +
 g_{w\bar z}\,dw\,d\bar z
 +
 g_{z\bar w}\,dz\,d\bar w
 +
 g_{z\bar z}\,dz\,d\bar z.
\end{equation*}
A horizontal vector has the form $\partial_w+\alpha\partial_z$. Orthogonality
to $\partial_z$ gives $g_{w\bar z}+\alpha g_{z\bar z}=0$, hence
$$ \alpha=-\frac{g_{w\bar z}}{g_{z\bar z}}. $$
Since $\Theta(\partial_w+\alpha\partial_z)=B+\alpha$, the condition $\Theta|_H=0$ gives
\begin{equation*}
 B=\frac{g_{w\bar z}}{g_{z\bar z}}.
\end{equation*}

For an affine element $\gamma$ as in \eqref{eq:general-affine-element},
\begin{equation*}
 d(a_\gamma w+b_\gamma)=a_\gamma\,dw,
 \qquad
 d(\lambda_\gamma z+h_\gamma(w))
 =
 \lambda_\gamma\,dz+h_\gamma'(w)\,dw.
\end{equation*}
Thus
$$
\gamma^*g_0=|a_\gamma|^2|dw|^2+|\lambda_\gamma dz+h_\gamma'(w)dw|^2.
$$
Expanding the second term gives
\begin{align*}
\gamma^* g_0 &=
|a_\gamma|^2|dw|^2 + |\lambda_\gamma dz+h_\gamma'(w)dw|^2 \\
&=
|a_\gamma|^2|dw|^2 + |\lambda_\gamma|^2 dz\,d\bar z
+h_\gamma'(w)\overline{\lambda_\gamma}\,dw\,d\bar z
+\lambda_\gamma\overline{h_\gamma'(w)}\,dz\,d\bar w
+|h_\gamma'(w)|^2dw\,d\bar w.
\end{align*}
Hence
\begin{equation*}
 g_{z\bar z}
 =
 \sum_{\gamma\in\Gamma}\chi(\gamma p)|\lambda_\gamma|^2,
 \qquad
 g_{w\bar z}
 =
 \sum_{\gamma\in\Gamma}\chi(\gamma p)h_\gamma'(w)\overline{\lambda_\gamma}.
\end{equation*}
Consequently,
\begin{equation*}
 B(w,z,\bar w,\bar z)
 =
 \frac{
 \displaystyle\sum_{\gamma\in\Gamma}
 \chi(\gamma(w,z))h_\gamma'(w)\overline{\lambda_\gamma}
 }{
 \displaystyle\sum_{\gamma\in\Gamma}
 \chi(\gamma(w,z))|\lambda_\gamma|^2
 }.
\end{equation*}
The denominator is strictly positive by \eqref{eq:cutoff-sum}, since $\lambda_\gamma\ne0$. Hence $B$ is smooth; in general it is neither holomorphic nor a function of $w$ alone.

\begin{lema} \label{lem:Theta-equivariant}
For every $\gamma\in\Gamma$, one has $\gamma^*\Theta=\lambda_\gamma\Theta$. In particular, $T^*\Theta=\Theta$, $G^*\Theta=\lambda\Theta$, and $\varphi^*\Theta=\xi\Theta$ when $r>1$.
\end{lema}

\begin{proof}
The kernel of $\Theta$ is $H$. Since $H$ is $\Gamma$-invariant,
$\gamma^*\Theta$ has the same kernel as $\Theta$. Hence it is a scalar
multiple of $\Theta$. Evaluating on $\partial_z$, we find
\begin{equation*}
 (\gamma^*\Theta)(\partial_z)
 =
 \Theta(d\gamma(\partial_z))
 =
 \Theta(\lambda_\gamma\partial_z)
 =
 \lambda_\gamma,
\end{equation*}
proving the statement.
\end{proof}

\subsection{The explicit normalised flow on the open surface}

Set
\begin{equation}\label{eq:mu-def}
 \mu:=-\frac{2\log|\lambda|}{\log k}.
\end{equation}
For $r>1$, replacing $G$ by $\varphi^jG$ changes its vertical multiplier from $\lambda$ to $\xi^j\lambda$; hence $|\lambda|$, and therefore $\mu$, is independent of this choice. By definition, $k^\mu|\lambda|^2=1$.
By Lemma~\ref{lem:Theta-equivariant} and \eqref{eq:s-transforms},
\begin{align*}
 G^*\left(s^\mu\sqrt{-1}\,\Theta\wedge\overline\Theta\right)
 &=
 (ks)^\mu\sqrt{-1}\,\lambda\Theta\wedge
 \bar\lambda\,\overline\Theta \\
 &=
 k^\mu|\lambda|^2s^\mu
 \sqrt{-1}\,\Theta\wedge\overline\Theta \\
 &=
 s^\mu\sqrt{-1}\,\Theta\wedge\overline\Theta.
\end{align*}
The same form is $T$-invariant. If $r>1$, then $\varphi^*s=s$ and $\varphi^*\Theta=\xi\Theta$ with $|\xi|=1$, so it is also $\varphi$-invariant. Hence it is invariant under the full deck group $\Gamma$ and descends to $S\setminus D$.
Assume from now on that $0<\mu<2$. For $M>0$, define
$$
A(t):=2-\mu+e^{-t}\bigl(M-(2-\mu)\bigr)=e^{-t}M+(1-e^{-t})(2-\mu).
$$
Thus $A(t)>0$ for all $t\geq0$, because $M>0$ and $2-\mu>0$.
Define the time-dependent $(1,1)$-form
\begin{equation}\label{eq:Omega-flow-def}
 \Omega_t^M
 :=
 A(t)\omega_{\rm tr}
 +
 e^{-t}s^\mu\sqrt{-1}\,\Theta\wedge\overline\Theta.
\end{equation}
The preceding invariance calculation shows that $\Omega_t^M$ descends from
$X$ to $S\setminus D$.

\begin{theorem}\label{thm:main-open-flow}
Let $S$ be a Kato surface of intermediate type, let $D$ be its maximal reduced divisor of rational curves, and assume $0<\mu<2$,  namely $k^{-1}<|\lambda|<1$.
Let $\Theta=dz+B\,dw$ be any
smooth $\Gamma$-equivariant vertical $(1,0)$-form on $X$. 
For every $M>0$, the form $\Omega_t^M$ on $S\setminus D$ defined in \eqref{eq:Omega-flow-def} is a smooth Hermitian metric on $S\setminus D$. It solves the normalised Chern--Ricci flow
\begin{equation*}
 \partial_t\Omega_t^M=-\operatorname{Ric}(\Omega_t^M)-\Omega_t^M
\end{equation*}
with initial value $\Omega_0^M = M\omega_{\rm tr} + s^\mu\sqrt{-1}\,\Theta\wedge\overline\Theta$.
Furthermore,
\begin{equation*}
 \Omega_t^M
 \longrightarrow
 (2-\mu)\omega_{\rm tr}
 \quad
 \text{in } \mathcal C^\infty_{\rm loc}(S\setminus D).
\end{equation*}
The limiting form is semipositive of rank one, and its kernel is the canonical foliation $\mathcal F$ considered above.
\end{theorem}

\begin{proof}
By construction, $\Omega_t^M$ is smooth and $\Gamma$-invariant on $X$;
hence it descends to a smooth $(1,1)$-form on $S\setminus D$. To prove positivity, let $v\in T_X^{1,0}$. Then
\begin{equation*}
 -\sqrt{-1}\,\Omega_t^M(v,\bar v)
 =
 A(t)\frac{|dw(v)|^2}{4s^2}
 +
 e^{-t}s^\mu|\Theta(v)|^2.
\end{equation*}
Both coefficients on the right-hand side are positive. If the right-hand side vanishes, then
$dw(v)=0$ and $\Theta(v)=0$. The first equality says that $v$ is
vertical, so $v=\alpha\partial_z$; the second gives
$\alpha=\Theta(v)=0$. Thus $v=0$, and $\Omega_t^M$ is positive definite.

We compute the determinant. Writing $\Theta=dz+Bdw$, one has
$$
\sqrt{-1}\,\Theta\wedge\overline\Theta=\sqrt{-1}\,dz\wedge d\bar z+\sqrt{-1}\,\bar B\,dz\wedge d\bar w+\sqrt{-1}\,B\,dw\wedge d\bar z+\sqrt{-1}\,|B|^2dw\wedge d\bar w.
$$
Hence the Hermitian matrix of $\Omega_t^M$ in the coframe $(dw,dz)$ is
\begin{equation*}
 \begin{pmatrix}
 \displaystyle \frac{A(t)}{4s^2}+e^{-t}s^\mu |B|^2
 &
 \displaystyle e^{-t}s^\mu B
 \\
 \displaystyle e^{-t}s^\mu\bar B
 &
 \displaystyle e^{-t}s^\mu
 \end{pmatrix}.
\end{equation*}
Hence
\begin{align*}
 \det(\Omega_t^M)
 &=
 \left(
 \frac{A(t)}{4s^2}
 +
 e^{-t}s^\mu|B|^2
\right)e^{-t}s^\mu
 -
 e^{-2t}s^{2\mu}|B|^2 \\
 &=
 \frac{A(t)e^{-t}}4s^{\mu-2}.
\end{align*}
The $B$-terms cancel exactly.
The Chern--Ricci form is
\begin{align*}
\operatorname{Ric}(\Omega_t^M)
 &= -\ddbar\log\det(\Omega_t^M) \\
 &= -\sqrt{-1}\partial\overline\partial \left( \log\left(\frac{A(t)e^{-t}}4\right) + (\mu-2)\log s\right) \\
 &= -\ddbar\bigl((\mu-2)\log s\bigr) \\
 &= -(\mu-2)\ddbar\log s \\
 &= -(2-\mu)\omega_{\rm tr}.
\end{align*}

We verify the evolution equation. From the definition of $A(t)$,
\begin{equation*}
 A'(t)
 =
 -e^{-t}\bigl(M-(2-\mu)\bigr)
 =
 (2-\mu)-A(t).
\end{equation*}
Since $\omega_{\rm tr}$, $s$, and $\Theta$ are independent of $t$,
\begin{equation*}
 \partial_t\Omega_t^M
 =
 A'(t)\omega_{\rm tr}
 -
 e^{-t}s^\mu\sqrt{-1}\,\Theta\wedge\overline\Theta.
\end{equation*}
On the other hand,
\begin{align*}
 -\operatorname{Ric}(\Omega_t^M)-\Omega_t^M
 &=
 (2-\mu)\omega_{\rm tr}
 -
 A(t)\omega_{\rm tr}
 -
 e^{-t}s^\mu\sqrt{-1}\,\Theta\wedge\overline\Theta \\
 &=
 \bigl((2-\mu)-A(t)\bigr)\omega_{\rm tr}
 -
 e^{-t}s^\mu\sqrt{-1}\,\Theta\wedge\overline\Theta \\
 &= A'(t)\omega_{\rm tr} - e^{-t} s^\mu \sqrt{-1}\,\Theta\wedge\overline\Theta \\
 &= \partial_t\Omega_t^M.
\end{align*}
The initial value follows from \eqref{eq:Omega-flow-def}, because $A(0)=M$.

Finally, let $K\Subset S\setminus D$ be compact. On finitely many lifted
coordinate charts covering $K$, the functions $s$, $B$, and all their
derivatives are bounded. Hence
\begin{equation*}
 e^{-t}s^\mu\sqrt{-1}\,\Theta\wedge\overline\Theta
 \longrightarrow 0
 \qquad
 \text{in }\mathcal C^\infty(K).
\end{equation*}
Since $A(t)\to2-\mu$,
\begin{equation*}
 \Omega_t^M
 \longrightarrow
 (2-\mu)\omega_{\rm tr}
 \qquad
 \text{in }\mathcal C^\infty_{\rm loc}(S\setminus D).
\end{equation*}
The limiting form is positive in the $w$-direction and vanishes on
$\partial_z$. Hence it has rank one and kernel equal to the vertical
foliation.
\end{proof}

 \begin{remark}
As $\mu\to2^-$ ({\itshape i.e.}\ $|\lambda|\to k^{-1}$), the limiting length
$$
L_G=\sqrt{\frac{2-\mu}{2}}\,\log k
$$
tends to zero, so the limiting circle degenerates to a point. As $\mu\to0^+$ ({\itshape i.e.}\ $|\lambda|\to1^-$), the limiting geometry remains non-degenerate. On the other hand, the estimate in \Cref{lem:green-cocycle-growth} contains the factor
$$
\frac{1}{1-|\lambda|},
$$
and therefore degenerates as $|\lambda|\to1^-$.
\end{remark}

We determine when the explicit ansatz is Gauduchon. Write $f=s^\mu$.
Since $\omega_{\rm tr}=\ddbar(-\log s)$, the transverse term is
$\partial\overline{\partial}$-closed. Hence the Gauduchon condition depends
only on the vertical term. Set
\begin{equation*}
\mathcal Q_B
:=
f_{w\bar w}
+
(f|B|^2)_{z\bar z}
-
(fB)_{z\bar w}
-
(f\bar B)_{w\bar z}.
\end{equation*}

\begin{theorem}
\label{thm:gauduchon-criterion-explicit-ansatz}
Let $S$ be a Kato surface of intermediate type, let $D$ be its maximal reduced divisor of rational curves, and assume $0<\mu<2$. For $M>0$, the metric $\Omega_t^M$ defined in \eqref{eq:Omega-flow-def} is Gauduchon for some, equivalently for every, $t\geq0$
if and only if $\mathcal Q_B=0$. If $B$ is independent of $z$ and $\bar z$, this condition reduces to $(s^\mu)_{w\bar w}=0$.
Under the standing assumption $0<\mu<2$, this is equivalent to
\begin{equation*}
\mu=1,
\qquad\text{or equivalently}\qquad
|\lambda|=k^{-1/2}.
\end{equation*}
\end{theorem}

\begin{proof}
Since $A(t)$ is independent of the spatial variables and $\partial\overline{\partial}\omega_{\rm tr}=0$,
\begin{equation*}
\partial\overline{\partial}\Omega_t^M
=
e^{-t}
\partial\overline{\partial}
\left(
\sqrt{-1}\,f\,\Theta\wedge\overline{\Theta}
\right).
\end{equation*}
As $e^{-t}>0$, the Gauduchon condition is equivalent, for every $t$, to
\begin{equation*}
\partial\overline{\partial}
\left(
\sqrt{-1}\,f\,\Theta\wedge\overline{\Theta}
\right)
=0.
\end{equation*}
The vertical term can be expanded as before:
in the coframe $(dw,dz)$, the Hermitian coefficient matrix is
\begin{equation*}
f
\begin{pmatrix}
|B|^2 & B\\
\bar B & 1
\end{pmatrix}.
\end{equation*}
Thus $h_{w\bar w}=f|B|^2$, $h_{w\bar z}=fB$, $h_{z\bar w}=f\bar B$, and $h_{z\bar z}=f$.
For completeness, let
\begin{equation*}
\beta=\sqrt{-1}\,h_{j\bar k}\,d\xi^j\wedge d\bar\xi^k,\qquad (\xi^1,\xi^2)=(w,z).
\end{equation*}
With the ordered volume form $\Xi:=dw\wedge dz\wedge d\bar w\wedge d\bar z$, a direct calculation gives
\begin{align*}
\partial\overline{\partial}
\left(\sqrt{-1}\,h_{w\bar w}\,dw\wedge d\bar w\right)
&=
-\sqrt{-1}\,(h_{w\bar w})_{z\bar z}\,\Xi,\\
\partial\overline{\partial}
\left(\sqrt{-1}\,h_{z\bar z}\,dz\wedge d\bar z\right)
&=
-\sqrt{-1}\,(h_{z\bar z})_{w\bar w}\,\Xi,\\
\partial\overline{\partial}
\left(\sqrt{-1}\,h_{w\bar z}\,dw\wedge d\bar z\right)
&=
\sqrt{-1}\,(h_{w\bar z})_{z\bar w}\,\Xi,\\
\partial\overline{\partial}
\left(\sqrt{-1}\,h_{z\bar w}\,dz\wedge d\bar w\right)
&=
\sqrt{-1}\,(h_{z\bar w})_{w\bar z}\,\Xi.
\end{align*}
Summing the four contributions gives
\begin{align*}
\partial\overline{\partial}\beta
&=
-\sqrt{-1}
\left(
(h_{z\bar z})_{w\bar w}
+
(h_{w\bar w})_{z\bar z}
-
(h_{w\bar z})_{z\bar w}
-
(h_{z\bar w})_{w\bar z}
\right)\Xi.
\end{align*}
Substituting the coefficients of
$\sqrt{-1}\,f\,\Theta\wedge\overline{\Theta}$, we obtain
\begin{align*}
\partial\overline{\partial}
\left(
\sqrt{-1}\,f\,\Theta\wedge\overline{\Theta}
\right)
&=
-\sqrt{-1}
\left(
f_{w\bar w}
+
(f|B|^2)_{z\bar z}
-
(fB)_{z\bar w}
-
(f\bar B)_{w\bar z}
\right)\Xi.
\end{align*}
Hence
\begin{equation*}
\partial\overline{\partial}
\left(
\sqrt{-1}\,f\,\Theta\wedge\overline{\Theta}
\right)
=0
\end{equation*}
if and only if
$$
f_{w\bar w}+(f|B|^2)_{z\bar z}-(fB)_{z\bar w}-(f\bar B)_{w\bar z}=0.
$$
This proves the first assertion. 
Suppose that $B$ is independent of $z$ and
$\bar z$. Since $f=s^\mu$ also depends only on $w,\bar w$, one has
\begin{equation*}
(f|B|^2)_{z\bar z}=0,
\qquad
(fB)_{z\bar w}=0,
\qquad
(f\bar B)_{w\bar z}=0.
\end{equation*}
Thus the Gauduchon equation reduces to $f_{w\bar w}=0$. Now
\begin{equation*}
s=-\Re w=-\frac{w+\bar w}{2},
\end{equation*}
so $s_w=s_{\bar w}=-\tfrac12$.
Since $f=s^\mu$, we obtain $f_w=-\frac{\mu}{2}s^{\mu-1}$.
Differentiation with respect to $\bar w$ gives
\begin{align*}
f_{w\bar w}
&=
-\frac{\mu}{2}(\mu-1)s^{\mu-2}s_{\bar w}\\
&=
\frac{\mu(\mu-1)}{4}s^{\mu-2}.
\end{align*}
Hence $f_{w\bar w}=0$ if and only if $\mu=0$ or $\mu=1$.
Under the standing assumption $0<\mu<2$, this is equivalent to $\mu=1$. Since
\begin{equation*}
\mu=-\frac{2\log|\lambda|}{\log k},
\end{equation*}
the identity $\mu=1$ is equivalent to $-2\log|\lambda|=\log k$, and hence to $|\lambda|=k^{-1/2}$.
This proves the second assertion.
\end{proof}

\begin{example}[Formal linear affine model]\label{ex:formal-linear-green}
For comparison, set $h=0$ in the two-generator affine model, so that $T(w,z)=(w+2\pi\sqrt{-1},z)$ and $G(w,z)=(kw,\lambda z)$. This formal linear model is useful for the Gauduchon equation, although it is not an intermediate Kato model in the Dloussky--Oeljeklaus--Toma normal form, where the polynomial term has positive degree \cite[Section~3]{oeljeklaus-toma-moduli}. If $0<\mu<1$ and $\vartheta\in\mathbb R$ is chosen so that $e^{\sqrt{-1}\vartheta\log k}=\lambda/\bar\lambda$, then
$$B=c(w,\bar w)\bar z,\qquad c(w,\bar w)=\frac{\sqrt{\mu(1-\mu)}}{2s}e^{\sqrt{-1}\vartheta\log s}$$
defines an equivariant vertical coform satisfying the Gauduchon equation.
\end{example}

\begin{proof}
Assume $h=0$, so that $T(w,z)=(w+2\pi\sqrt{-1},z)$ and $G(w,z)=(kw,\lambda z)$. Choose $\vartheta\in\mathbb{R}$ such that $e^{\sqrt{-1}\vartheta\log k}=\lambda/\bar\lambda$. This is possible because $\lambda/\bar\lambda\in S^1$. Let $B=c(w,\bar w)\bar z$, where
\begin{equation*}
c(w,\bar w)
=
\frac{\sqrt{\mu(1-\mu)}}{2s}
e^{\sqrt{-1}\vartheta\log s}.
\end{equation*}
The condition $0<\mu<1$ guarantees that $\sqrt{\mu(1-\mu)}$ is real and
positive. Since $s>0$, the function $\log s$ is smooth.
First consider $T$. Since $T$ preserves $s$, $dw$, $dz$, and $\bar z$, one has $T^*\Theta=\Theta$.
For $G$, we compute
\begin{align*}
G^*\Theta
&=
d(\lambda z)
+
B(kw,\lambda z,k\bar w,\bar\lambda\bar z)\,d(kw)\\
&=
\lambda dz
+
k\,B(kw,\lambda z,k\bar w,\bar\lambda\bar z)\,dw.
\end{align*}
Since $s(kw)=ks(w)$,
\begin{align*}
c(kw,k\bar w)
&=
\frac{\sqrt{\mu(1-\mu)}}{2ks}
e^{\sqrt{-1}\vartheta\log(ks)}\\
&=
\frac{1}{k}
e^{\sqrt{-1}\vartheta\log k}
c(w,\bar w) \\
&= \frac{\lambda}{k\bar\lambda}c(w,\bar w).
\end{align*}
Hence
\begin{align*}
k\,B(kw,\lambda z,k\bar w,\bar\lambda\bar z)
&=
k\,c(kw,k\bar w)\bar\lambda\bar z\\
&=
\lambda c(w,\bar w)\bar z\\
&=
\lambda B(w,z,\bar w,\bar z).
\end{align*}
It follows that $G^*\Theta=\lambda dz+\lambda B\,dw=\lambda\Theta$, so $\Theta$ is equivariant. It remains to check the Gauduchon equation. Since $B=c\bar z$, we have
$
|B|^2=|c|^2|z|^2.
 $
The functions $f$ and $c$ depend only on $w,\bar w$. Hence
\begin{equation*}
(f|B|^2)_{z\bar z}
=
(f|c|^2|z|^2)_{z\bar z}
=
f|c|^2.
\end{equation*}
Also, $fB=fc\bar z$ is independent of $z$, whereas $f\bar B=f\bar c z$ is independent of $\bar z$. Hence $(fB)_{z\bar w}=(f\bar B)_{w\bar z}=0$, and the Gauduchon equation reduces to $f_{w\bar w}+f|c|^2=0$.
As computed above,
\begin{equation*}
f_{w\bar w}
=
\frac{\mu(\mu-1)}{4}s^{\mu-2}.
\end{equation*}
Moreover,
\begin{equation*}
|c|^2
=
\frac{\mu(1-\mu)}{4s^2}.
\end{equation*}
Hence
\begin{equation*}
f|c|^2
=
s^\mu\frac{\mu(1-\mu)}{4s^2}
=
\frac{\mu(1-\mu)}{4}s^{\mu-2}.
\end{equation*}
Adding the two terms gives
\begin{align*}
f_{w\bar w}+f|c|^2
&=
\frac{\mu(\mu-1)}{4}s^{\mu-2}
+
\frac{\mu(1-\mu)}{4}s^{\mu-2}=0.
\end{align*}
Thus
$$
\partial\overline{\partial}\!\left(\sqrt{-1}\,s^\mu\Theta\wedge\overline{\Theta}\right)=0.
$$
This proves the claim.
\end{proof}

\begin{remark}
For the general transformation $G(w,z)=(kw,\lambda z+h(w))$, the equivariance equations for $B$ and the Gauduchon equation are independent requirements; equivariance alone does not imply the Gauduchon condition.
More precisely, the condition $T^*\Theta=\Theta$ imposes the periodicity condition
\begin{equation*}
B(w+2\pi\sqrt{-1},z,\bar w-2\pi\sqrt{-1},\bar z)
=
B(w,z,\bar w,\bar z).
\end{equation*}
The condition $G^*\Theta=\lambda\Theta$ becomes
\begin{align*}
h'(w)
+
k\,B
\bigl(
kw,
\lambda z+h(w),
k\bar w,
\bar\lambda\bar z+\overline{h(w)}
\bigr)
=
\lambda B(w,z,\bar w,\bar z).
\end{align*}
These are the equivariance equations for the index-one subgroup. For $r>1$, one must in addition impose $\varphi^*\Theta=\xi\Theta$; this is automatic for the coform constructed in \S2.2, since the averaging is taken over the full deck group. These equivariance conditions are distinct from the Gauduchon
equation
\begin{equation*}
f_{w\bar w}
+
(f|B|^2)_{z\bar z}
-
(fB)_{z\bar w}
-
(f\bar B)_{w\bar z}
=
0.
\end{equation*}
Thus, in the general case, a smooth coefficient $B$ must satisfy two
separate types of requirements: the equivariance equations, which are needed
for descent to the quotient, and the Gauduchon equation, which is needed for
the descended Hermitian form to be Gauduchon. The former do not imply the latter.
\end{remark}

Even when the equivariant family is not Gauduchon for finite time, its
Gauduchon defect decays exponentially. Indeed,
\begin{align*}
\partial\overline{\partial}\Omega_t^M
&=
\partial\overline{\partial}
\left(
A(t)\omega_{\rm tr}
+
e^{-t}\sqrt{-1}\,s^\mu\Theta\wedge\overline{\Theta}
\right)\\
&=
e^{-t}
\partial\overline{\partial}
\left(
\sqrt{-1}\,s^\mu\Theta\wedge\overline{\Theta}
\right),
\end{align*}
because $A(t)$ is independent of the space variables and $\partial\overline{\partial}\omega_{\rm tr}=0$.
Consequently, for every compact $K\Subset S\setminus D$ and every $m\geq0$, the $\mathcal C^m(K)$-norm of $\partial\overline\partial\Omega_t^M$ is $O(e^{-t})$ with respect to any fixed background metric. This decay does not imply the exact Gauduchon condition at finite time.

The leafwise geometry of the explicit solution is recorded in the following proposition.
\begin{propo}
\label{prop:hyperbolic-leafwise-flatness}
Let $S$ be a Kato surface of intermediate type, let $D$ be its maximal reduced divisor of rational curves, and assume $0<\mu<2$. For $M>0$, the family $\Omega_t^M$ on $S\setminus D$ defined in \eqref{eq:Omega-flow-def} is flat along the leaves of the canonical foliation on $S\setminus D$ and collapses homothetically along them, with metric factor $e^{-t}$.
\end{propo}

\begin{proof}
The one-form $dw$ is semi-invariant: $T^*dw=dw$ and $G^*dw=k\,dw$; when $r>1$, one also has $\varphi^*dw=dw$. Hence its kernel is invariant under the full deck group. The canonical
foliation is induced, on the cover $\mathbb H_\ell\times\mathbb C$, by $\ker dw=\mathbb C\,\partial_z$. Its lifted leaves are $L_{w_0}:=\{w=w_0\}\times\mathbb C$, where $w_0=-s_0+\sqrt{-1}\theta_0$ and $s_0:=-\Re w_0>0$.
Along $L_{w_0}$ one has $dw=ds=d\theta=0$, and hence $\omega_{\rm tr}|_{L_{w_0}}=0$.
Since $\Theta=dz+B\,dw$, one has $\Theta|_{L_{w_0}}=dz$.
Restricting the explicit metric $\Omega_t^M = A(t)\omega_{\rm tr} + e^{-t}s^\mu\sqrt{-1}\Theta\wedge\overline{\Theta}$ to $L_{w_0}$, we obtain
$$
\Omega_t^M|_{L_{w_0}}
=
\sqrt{-1}\,e^{-t}s_0^\mu\,dz\wedge d\bar z.
$$
This is a constant multiple of the Euclidean metric on the $z$-line.
Consequently the intrinsic Chern--Ricci form vanishes:
$$
\operatorname{Ric}\bigl(\Omega_t^M|_{L_{w_0}}\bigr)
=
-\sqrt{-1}\,\partial_z\bar\partial_z
\log(e^{-t}s_0^\mu)
=
0.
$$
Thus the leafwise metrics are flat. Since the only time-dependent factor in
the leafwise metric is $e^{-t}$, the collapse along the leaves is homothetic,
with metric factor $e^{-t}$, equivalently with length factor $e^{-t/2}$.
\end{proof}

\begin{remark}
This description is compatible with the affine Green action. The generator $T$ preserves $s_0$ and $dz$. The generator $G$ sends $L_{w_0}$ to
$L_{kw_0}$, and
$$
s(kw_0)=ks_0,
\qquad
dz\longmapsto \lambda\,dz.
$$
Since $|\lambda|^2=k^{-\mu}$, one has
$$
e^{-t}(ks_0)^\mu|\lambda|^2|dz|^2
=
e^{-t}s_0^\mu|dz|^2.
$$
The additional transformation $\varphi$ preserves $s$ and sends $dz$ to $\xi\,dz$; as $|\xi|=1$, it also preserves the leafwise metric. Hence the leafwise expression is equivariant under the full deck group and descends to the leaves on $S\setminus D$.
\end{remark}

\subsection{Gromov--Hausdorff collapse of the Chern--Ricci flow}

With the notation above, set $\rho:=\frac12\log s$.
The transformations $T$ and, when $r>1$, $\varphi$ preserve $\rho$, whereas $G$ sends $\rho$ to $\rho+(1/2)\log k$. Set
$$
\ell_G:=\frac12\log k,
\qquad
S^1_G:=\mathbb{R}/\ell_G\mathbb{Z},
$$
and denote by $\bar\rho:A\longrightarrow S^1_G$ the induced Green projection.

Let $g_t$ be the Riemannian metric associated with $\Omega_t^M$, and let
$d_t$ be its distance on $A$. If $K\Subset A$, the notation
$d_t|_{K\times K}$ always denotes the ambient distance of $A$ restricted to
$K\times K$, not the intrinsic length distance of $K$. We endow $S^1_G$
with the length metric induced by
$$
2(2-\mu)\,d\rho^2,
$$
and denote the corresponding distance by $d_G$. Its total length is
$$ L_G=\sqrt{2(2-\mu)}\,\ell_G=\sqrt{\frac{2-\mu}{2}}\,\log k.$$
The Riemannian metric associated with $\omega_{\rm tr}$ is
\begin{equation*}
g_{\rm tr}
=
\frac{1}{2s^2}(ds^2+d\theta^2).
\end{equation*}
Since $d\rho=(1/2)\,ds/s$, we have $ds=2s\,d\rho$, and hence
\begin{equation*}
g_{\rm tr}
=
2\,d\rho^2+\frac{1}{2s^2}d\theta^2.
\end{equation*}
Hence $g_t = A(t)g_{\rm tr} + 2e^{-t}s^\mu|\Theta|^2$, with $A(t)\longrightarrow 2-\mu$ as $t\to+\infty$.
In particular,
$$
g_t\geq A(t)g_{\rm tr}\geq 2A(t)\,d\rho^2.
$$
Let $\gamma:[0,1]\to A$ be a piecewise smooth path from $p$ to $q$. Fix lifts $\tilde p,\tilde q\in X$ of $p,q$, and lift $\gamma$ with $\widetilde\gamma(0)=\tilde p$. Then $\widetilde\gamma(1)=\delta\tilde q$ for some $\delta\in\Gamma$. The function $\rho=\frac12\log s$ is real-valued on $X$, and the preceding pointwise inequality gives
\begin{align*}
\operatorname{Length}_{g_t}(\gamma)
&\geq
\sqrt{2A(t)}
\int_0^1|\dot\rho(\widetilde\gamma(\tau))|\,d\tau.
\end{align*}
Every deck transformation changes $\rho$ by an integral multiple of $\ell_G$: the generators $T$ and $\varphi$ preserve $\rho$, whereas $G$ sends $\rho$ to $\rho+\ell_G$. Consequently,
$$
\int_0^1|\dot\rho(\widetilde\gamma(\tau))|\,d\tau
\geq
\min_{m\in\mathbb Z}|\rho(\tilde p)-\rho(\tilde q)-m\ell_G|
=
\frac{d_G(\bar\rho(p),\bar\rho(q))}{\sqrt{2(2-\mu)}}.
$$
Taking the infimum over all paths from $p$ to $q$, we obtain
\begin{equation}
\label{eq:GH-lower-bound}
d_t(p,q)
\geq
\left(\sqrt{\frac{A(t)}{2-\mu}}\right)
d_G(\bar\rho(p),\bar\rho(q)).
\end{equation}
This is the intermediate analogue of the fact that the projection to the base is
distance non-increasing in the Enoki case. Thus the Green direction does not collapse.

\phantomsection
\label{def:compact-green-block}
Fix $s_0>0$, $R>0$, and set
$$
\rho_*:=\frac12\log s_0,
\qquad
I_*:=[\rho_*,\rho_*+\ell_G],
\qquad
I_*^+:=[\rho_*-\ell_G,\rho_*+2\ell_G].
$$
On the index-one cover $A_1$, the quotient of the real slice $\theta=0$ by the action generated by $G$ is a smooth affine $\mathbb C$-bundle over $S^1_G$. Since its fibre $\mathbb C$ is contractible, it admits a smooth section $\sigma_1:S^1_G\to A_1$. Let $p:A_1\to A$ be the finite cyclic covering of \Cref{prop:higher-index-green-model}. Since $p$ preserves the Green coordinate, $\bar\rho\circ p=\bar\rho_1$, where $\bar\rho_1:A_1\to S^1_G$ denotes the Green projection. Hence $\sigma:=p\circ\sigma_1$ satisfies $\bar\rho\circ\sigma=\operatorname{id}_{S^1_G}$ and is a smooth section of $\bar\rho$. Choose an equivariant lift over $\mathbb R$ of the form
$$\widetilde\sigma(\rho)=\bigl(-e^{2\rho},z_\sigma(\rho)\bigr).$$
Equivariance of the lift is exactly
\begin{equation}\label{eq:z-sigma-equivariance}
z_\sigma(\rho+\ell_G)=\lambda z_\sigma(\rho)+h(-e^{2\rho}),
\end{equation}
so that $\widetilde\sigma(\rho+\ell_G)=G(\widetilde\sigma(\rho))$.

Choose $R^+>R$ large enough so that the graph of
$\widetilde\sigma$ over $I_*^+$ is contained in $\{|z|\leq R^+\}$. On
the Green cover define
\begin{equation*}
\widetilde K_{s_0,R}
:=
\{s_0\leq s\leq ks_0,\ 0\leq\theta\leq2\pi,\ |z|\leq R\},
\end{equation*}
and
\begin{equation*}
\widetilde K^+_{s_0,R^+}
:=
\{k^{-1}s_0\leq s\leq k^2s_0,\ -1\leq\theta\leq2\pi+1,\ |z|\leq R^+\}.
\end{equation*}
The compact subsets
$$
K_{s_0,R}:=q(\widetilde K_{s_0,R}),
\qquad
K^+_{s_0,R^+}:=q(\widetilde K^+_{s_0,R^+})
$$
will be called a \emph{compact Green block} and its compact enlargement. For fixed $s_0>0$, the family $\{K_{s_0,R}\}_{R>0}$ exhausts $A$ as $R\to+\infty$. Indeed, every point of $A$ admits a lift to $X$. Applying a suitable power of $G$, we may arrange that its $s$-coordinate lies in $[s_0,ks_0]$, and then applying a suitable power of $T$ we may arrange that its angular coordinate lies in $[0,2\pi]$. The resulting representative has finite $z$-coordinate and therefore belongs to $\widetilde K_{s_0,R}$ for all sufficiently large $R$.

The annulus $s_0\leq s\leq ks_0$ represents one full period of
$S^1_G$, so $\bar\rho(K_{s_0,R})=S^1_G$. The larger tube
$\widetilde K^+_{s_0,R^+}$ provides the room needed for the estimates below.
Indeed, angular corrections of size $<1$ remain inside the interval
$-1\leq\theta\leq2\pi+1$, and every section arc realising a shortest
Green-circle distance can be lifted inside $I_*^+$. Since $\widetilde K^+_{s_0,R^+}$ is compact in $X$ and $B$ is smooth on $X$, the functions $s$, $s^{-1}$, $B$, and the relevant first derivatives of $B$ are uniformly bounded on $\widetilde K^+_{s_0,R^+}$.

Purely vertical segments are estimated by pulling back the metric to curves
$$
\beta(\tau)=(w_0,z(\tau))
$$
in the Green cover. Since $w_0$ is fixed, $\beta^*dw=0$, and hence
$$
\beta^*\Theta
=
\beta^*(dz+B\,dw)
=
d(z\circ\beta).
$$
Thus the transverse term $A(t)\omega_{\rm tr}$ vanishes on such segments,
and their lengths are controlled only by the vertical term
$2e^{-t}s(w_0)^\mu|\Theta|^2$. No bound for $B$ is needed there.

We record the following elementary estimate for the affine Green action. If
$\tilde p\in\mathbb{H}_\ell\times\mathbb{C}$, we denote by $z(\tilde p)$ its vertical
coordinate.

\begin{lema}
\label{lem:green-cocycle-growth}
Fix a compact Green block $K_{s_0,R}$ with enlargement
$K^+_{s_0,R^+}$. Set $a_\lambda:=-\log|\lambda|$. Then $a_\lambda>0$, and there exist constants $C_K>0$ and $b_K\geq0$ such that
the following holds. For $n\geq1$ and $m\in\mathbb{Z}$, let
$$
\gamma_{n,m}:=G^{-n}T^mG^n.
$$
If $\tilde p,\tilde q\in\mathbb{H}_\ell\times\mathbb{C}$ satisfy
\begin{equation*}
k^{-1}s_0\leq s(\tilde p),s(\tilde q)\leq k^2s_0,
\qquad
|z(\tilde p)|\leq R^+,\qquad |z(\tilde q)|\leq R^+,
\end{equation*}
then
\begin{equation*}
\bigl|
z(\gamma_{n,m}\tilde p)-z(\tilde q)
\bigr|
\leq
C_K\exp(a_\lambda n+b_Kk^n).
\end{equation*}
\end{lema}

\begin{proof}
Recall that $|\lambda|=k^{-\mu/2}$. Since $0<\mu<2$, one has $0<|\lambda|<1$, and hence $a_\lambda=-\log|\lambda|>0$. 
We first compute the conjugate explicitly. A direct induction gives
\begin{equation*}
G^n(w,z)
=
\left(
k^nw,\lambda^nz+H_n(w)
\right),
\end{equation*}
where
\begin{equation*}
H_n(w):=
\sum_{j=0}^{n-1}
\lambda^{\,n-1-j}h(k^jw).
\end{equation*}
Put
\begin{equation*}
\delta_{n,m}:=\frac{2\pi\sqrt{-1}m}{k^n}.
\end{equation*}
Then
$$
T^mG^n(w,z)=\left(k^n(w+\delta_{n,m}),\lambda^nz+H_n(w)\right).
$$
To apply $G^{-n}$, we solve
$$
G^n(w+\delta_{n,m},Z)=T^mG^n(w,z).
$$
The first coordinates already agree. Comparing the second coordinates gives
\begin{equation*}
\lambda^nZ+H_n(w+\delta_{n,m})
=
\lambda^nz+H_n(w),
\end{equation*}
and therefore
\begin{equation}
\label{eq:green-conjugate-formula}
\gamma_{n,m}(w,z)
=
\left(
w+\delta_{n,m},
z+\lambda^{-n}
\bigl(H_n(w)-H_n(w+\delta_{n,m})\bigr)
\right).
\end{equation}
In particular, $\gamma_{n,m}$ changes the $w$-coordinate by a purely
imaginary translation, and therefore it does not change $s=-\Re w$.

It remains to estimate the cocycle term. Let $w=w(\tilde p)$. On the index-one subgroup, the Dloussky--Oeljeklaus--Toma affine Green normal form gives a $T$-periodic cocycle of the form $h(w)=\widehat h(e^{-w})$, where $\widehat h$ is a finite normal polynomial \cite{dloussky-oeljeklaus-toma,oeljeklaus-toma-moduli}. If
$\widehat h\not\equiv0$, let $d:=\deg\widehat h$, and otherwise put
$d:=0$.
We have
$$
|e^{-w}|=e^{s(w)}.
$$
Thus there exists a constant $C_K>0$ such that
\begin{equation}
\label{eq:h-polynomial-growth}
|h(\zeta)|
\leq
C_K\exp(d\,s(\zeta))
\end{equation}
on the Green half-strips under consideration.

Now $\delta_{n,m}$ is purely imaginary. Hence adding
$\delta_{n,m}$ does not change the real part. For every
$0\leq j\leq n-1$, we have
\begin{equation*}
s(k^jw)=k^js(w),
\qquad
s(k^j(w+\delta_{n,m}))=k^js(w).
\end{equation*}
Since $s(w)\leq k^2s_0$ on the controlled enlargement and $j\leq n-1$, we
have, after increasing the constant if necessary,
$$
s(k^jw)\leq C_K k^n,
\qquad
s(k^j(w+\delta_{n,m}))\leq C_K k^n.
$$
 (Here and below, $C_K$ denotes a constant depending only on $K$
and on the section $\sigma$; its value may change from line to line.)
Hence, for some $b_K\geq0$,
\begin{equation}
\label{eq:h-growth-green-halfstrip}
|h(k^jw)|+
|h(k^j(w+\delta_{n,m}))|
\leq
C_K\exp(b_Kk^n)
\end{equation}
for every $0\leq j\leq n-1$. 
Using the definition of $H_n$, we obtain
\begin{align*}
\left|
H_n(w)-H_n(w+\delta_{n,m})
\right|
&\leq
\sum_{j=0}^{n-1}
|\lambda|^{\,n-1-j}
\left|
h(k^jw)-h(k^j(w+\delta_{n,m}))
\right| \\
&\leq
C_K\exp(b_Kk^n)
\sum_{j=0}^{n-1}
|\lambda|^{\,n-1-j}.
\end{align*}
Since $0<|\lambda|<1$,
\begin{equation*}
\sum_{j=0}^{n-1}
|\lambda|^{\,n-1-j}
=
\sum_{\ell=0}^{n-1}|\lambda|^\ell
\leq
\frac{1}{1-|\lambda|}.
\end{equation*}
Thus, after increasing $C_K$,
\begin{equation}
\label{eq:Hn-difference-growth-bound}
\left|
H_n(w)-H_n(w+\delta_{n,m})
\right|
\leq
C_K\exp(b_Kk^n).
\end{equation} 
Finally, using \eqref{eq:green-conjugate-formula},
\eqref{eq:Hn-difference-growth-bound}, and the bounds $|z(\tilde p)|,|z(\tilde q)|\leq R^+$, we obtain
\begin{align*}
\bigl|
z(\gamma_{n,m}\tilde p)-z(\tilde q)
\bigr|
&\leq
|z(\tilde p)-z(\tilde q)|
+
|\lambda|^{-n}
\left|
H_n(w)-H_n(w+\delta_{n,m})
\right| \\
&\leq
C_K+C_K|\lambda|^{-n}\exp(b_Kk^n).
\end{align*}
Since $|\lambda|^{-n}=e^{a_\lambda n}$ and $\exp(a_\lambda n+b_Kk^n)\geq1$, we may enlarge $C_K$ once more and obtain
$$
\bigl|
z(\gamma_{n,m}\tilde p)-z(\tilde q)
\bigr|
\leq
C_K\exp(a_\lambda n+b_Kk^n).
$$
This proves the lemma.
\end{proof}

We now establish the analogue of the fibre-diameter estimate in the Enoki case. There the fibres collapse because their metric is multiplied by $e^{-t}$. Here the $z$-direction collapses directly by the same factor, whereas the angular direction is treated using the conjugates $G^{-n}T^mG^n$, which produce angular translations of size $2\pi m/k^n$.

\begin{lema}
\label{lem:fibre-collapse}
Let $K=K_{s_0,R}$ be a compact Green block. Then there exists a non-negative function $\varepsilon_K$ on $[0,\infty)$, finite for every $t\ge0$ and satisfying $\varepsilon_K(t)\to0$ as $t\to+\infty$, such that
\begin{equation*}
d_t(p,\sigma(\bar\rho(p)))\leq \varepsilon_K(t)
\end{equation*}
for every $p\in K$. In particular, if $p,q\in K$ and
$\bar\rho(p)=\bar\rho(q)$, then
$
d_t(p,q)\leq 2\varepsilon_K(t).
 $
\end{lema}

\begin{proof}
Fix $p\in K$ and set $p^\sigma:=\sigma(\bar\rho(p))$.
Choose lifts $\tilde p\in\widetilde K_{s_0,R}$ and
$\tilde p^\sigma=\widetilde\sigma(\rho(\tilde p))$. Thus the two lifts have
the same $s$-coordinate. The remaining discrepancy is in the angular
coordinate and in the vertical coordinate.

For $n\geq1$, set $\gamma_{n,m}:=G^{-n}T^mG^n$.
By \Cref{lem:green-cocycle-growth}, this conjugate acts on the $w$-coordinate
as
$$
w\longmapsto w+\frac{2\pi\sqrt{-1}m}{k^n}.
$$
Set $\Delta\theta_0:=\theta(\tilde p^\sigma)-\theta(\tilde p)$.
The conjugate $\gamma_{n,m}$ changes the angular coordinate by
$$
\theta\longmapsto \theta+\frac{2\pi m}{k^n}.
$$
Choose $m=m(n)$ to be an integer nearest to
$$
\frac{k^n}{2\pi}\Delta\theta_0.
$$
Then
$$
\left|
m(n)-\frac{k^n}{2\pi}\Delta\theta_0
\right|
\leq
\frac12.
$$
Multiplying by $2\pi/k^n$, we obtain
$$
\left|
\frac{2\pi m(n)}{k^n}
-
\Delta\theta_0
\right|
\leq
\frac{\pi}{k^n}.
$$
Hence
$$
\left|
\theta( \gamma_{n,m(n)}\tilde p)-\theta(\tilde p^\sigma)
\right|
\leq
C_K k^{-n}.
$$
Since $\theta(\tilde p)\in[0,2\pi]$ and
$\theta(\tilde p^\sigma)=0$, the number $\Delta\theta_0$ is uniformly
bounded. Hence
$$
|m(n)|
\leq
\frac{k^n}{2\pi}|\Delta\theta_0|+\frac12
\leq
C_K k^n.
$$
By \Cref{lem:green-cocycle-growth}, the remaining vertical displacement
satisfies
\begin{equation}
\label{eq:vertical-displacement-bound}
\bigl|
z( \gamma_{n,m(n)}\tilde p)-z(\tilde p^\sigma)
\bigr|
\leq
C_K\exp(a_\lambda n+b_Kk^n).
\end{equation}

Since $ \gamma_{n,m(n)}\tilde p$ projects to the same point of $A$ as
$\tilde p$, it is enough to estimate a path from
$ \gamma_{n,m(n)}\tilde p$ to $\tilde p^\sigma$ in the cover. We use two
segments.
First, keep $w$ fixed and move only in the $z$-direction. Put
$
z_0:=z( \gamma_{n,m(n)}\tilde p), 
z_1:=z(\tilde p^\sigma),
$
and define
$$
\beta_1(\tau)
=
\bigl(w(\gamma_{n,m}\tilde p),(1-\tau)z_0+\tau z_1\bigr),
\qquad
0\leq\tau\leq1.
$$
Along $\beta_1$, one has $\beta_1^*dw=0$, hence $\beta_1^*\Theta=d(z\circ\beta_1)$. Therefore
$$
g_t(\dot\beta_1,\dot\beta_1)
=
2e^{-t}s^\mu|z_1-z_0|^2.
$$
Since $s\in[s_0,ks_0]$ along this segment,
\eqref{eq:vertical-displacement-bound} gives
\begin{align*}
\operatorname{Length}_{g_t}(\beta_1)
&\leq
C_K e^{-t/2}|z_1-z_0|\\
&\leq
C_K e^{-t/2}\exp(a_\lambda n+b_Kk^n).
\end{align*}

Second, keep the corrected $z$-coordinate fixed and correct the remaining
angular error. Set
$$
\Delta\theta
:=
\theta(\tilde p^\sigma)-\theta( \gamma_{n,m(n)}\tilde p).
$$
Then
$$
|\Delta\theta|\leq C_Kk^{-n}.
$$
For $n$ sufficiently large,  $n\ge N_K:=\lceil\log_kC_K\rceil+1$, this angular error is $<1$, so the segment below lies inside
the enlarged tube:
\begin{equation*}
\beta_2(\tau)
=
\bigl(
-s+\sqrt{-1}(\theta( \gamma_{n,m(n)}\tilde p)+\tau\Delta\theta),
z_1
\bigr),
\qquad
0\leq\tau\leq1,
\end{equation*}
where $s$ is the common $s$-coordinate. Along $\beta_2$, we have
$$
ds=0,
\qquad
dw=\sqrt{-1}\Delta\theta\,d\tau,
\qquad
dz=0,
$$
and hence
$$
\Theta(\dot\beta_2)=\sqrt{-1}B\,\Delta\theta.
$$
Also,
$$
g_{\rm tr}(\dot\beta_2,\dot\beta_2)
=
\frac{|\Delta\theta|^2}{2s^2}.
$$
Thus
\begin{align*}
g_t(\dot\beta_2,\dot\beta_2)
&=
|\Delta\theta|^2
\left(
\frac{A(t)}{2s^2}
+
2e^{-t}s^\mu |B|^2
\right).
\end{align*}
On $\widetilde K^+_{s_0,R^+}$, the functions $s$, $s^{-1}$, and $B$ are uniformly bounded, and $A(t)$ is uniformly bounded for $t\geq0$. Hence
$$
\operatorname{Length}_{g_t}(\beta_2)
\leq
C_K|\Delta\theta|
\leq
C_K k^{-n}.
$$

Combining the two segments, we obtain
\begin{equation*}
d_t(p,p^\sigma)
\leq
C_K\left(k^{-n}+e^{-t/2}\exp(a_\lambda n+b_Kk^n)\right).
\end{equation*}
Choose $\delta>0$ so small that $b_K\delta<1/2$,
and set
\[
t_0:=t_0(K):=\frac{k^{N_K}}{\delta},
\]
so that $t>t_0$ implies $n(t):=\lfloor\log_k(\delta t)\rfloor\ge N_K$. Then, for
$t>t_0$,
$$
k^{n(t)}\leq \delta t
\qquad\text{and}\qquad
k^{-n(t)}\leq \frac{k}{\delta t}\longrightarrow0.
$$
Moreover,
$$
e^{a_\lambda n(t)}
\leq
C_K t^{a_\lambda/\log k}.
$$
Hence
$$
e^{-t/2}\exp(a_\lambda n(t)+b_Kk^{n(t)})
\leq
C_K t^{a_\lambda/\log k}
\exp\bigl((-1/2+b_K\delta)t\bigr)
\longrightarrow0.
$$
Hence
$$
d_t(p,p^\sigma)\leq \varepsilon_K(t),
\qquad
\varepsilon_K(t)\to0,
$$
uniformly for $p\in K$. If $p,q\in K$ and
$\bar\rho(p)=\bar\rho(q)$, then $\sigma(\bar\rho(p))=\sigma(\bar\rho(q))$. The triangle inequality gives
$$
d_t(p,q)\leq 2\varepsilon_K(t).
$$
This proves the claim for $t>t_0$. For $t\in[0,t_0]$, set
\[
\varepsilon_K(t):=\sup_{p\in K}d_t\big(p,\sigma(\bar\rho(p))\big),
\]
which is finite since $K$ is compact.
\end{proof}

We shall also use the corresponding upper estimate in the Green direction.

\begin{lema}
\label{lem:green-upper-estimate}
Let $K=K_{s_0,R}$ be a compact Green block. Then
\begin{equation*}
d_t(p,q)
\leq
d_G(\bar\rho(p),\bar\rho(q))+o_K(1)
\end{equation*}
uniformly for $p,q\in K$.
\end{lema}

\begin{proof}
Set $\xi:=\bar\rho(p)$ and $\eta:=\bar\rho(q)$.
By \Cref{lem:fibre-collapse},
$$
d_t(p,\sigma(\xi))\leq \varepsilon_K(t),
\qquad
d_t(q,\sigma(\eta))\leq \varepsilon_K(t).
$$
Thus
$$
d_t(p,q)
\leq
2\varepsilon_K(t)+d_t(\sigma(\xi),\sigma(\eta)).
$$
Choose lifts $a,b\in I_*^+$ of $\xi,\eta$ such that the interval
$[a,b]$ realises the circle distance; explicitly, choose $a\in I_*$ to be a lift of $\xi$, and
$b$ the lift of $\eta$ with $|b-a|\le\ell_G/2$, so that $[a,b]$ realises
$d_G(\xi,\eta)$ and $b\in I^+_*$.
Then
$$
d_G(\xi,\eta)
=
\sqrt{2(2-\mu)}\,|b-a|.
$$
The lifted section arc is
$$
\Gamma(\rho)=\widetilde\sigma(\rho)
=
\bigl(-e^{2\rho},z_\sigma(\rho)\bigr),
\qquad
\rho\in[a,b].
$$
On $I_*^+$, we have $|z_\sigma'|\leq C_K$. Along $\Gamma$ one has
$$
d\theta=0,
\qquad
ds=2s\,d\rho,
\qquad
dw=-2s\,d\rho.
$$
Hence $g_{\rm tr}(\dot\Gamma,\dot\Gamma)=2$. Furthermore,
$$
\Theta(\dot\Gamma)
=
z_\sigma'(\rho)-2sB.
$$
Since $s$ and $B$ are uniformly bounded on $\widetilde K^+_{s_0,R^+}$, we obtain
$$
|\Theta(\dot\Gamma)|\leq C_K.
$$
Consequently,
$$
g_t(\dot\Gamma,\dot\Gamma)=2A(t)+2e^{-t}s^\mu|\Theta(\dot\Gamma)|^2\leq 2A(t)+C_K e^{-t}.
$$
Hence
$$
\operatorname{Length}_{g_t}(\Gamma)
\leq
\sqrt{2A(t)+C_Ke^{-t}}\,|b-a|.
$$
Since $A(t)\to2-\mu$, this gives
\begin{align*}
\operatorname{Length}_{g_t}(\Gamma)
&\leq
\sqrt{2(2-\mu)}\,|b-a|+o_K(1)\\
&=
d_G(\xi,\eta)+o_K(1),
\end{align*}
uniformly in $\xi,\eta$. Hence
$$
d_t(\sigma(\xi),\sigma(\eta))
\leq
d_G(\xi,\eta)+o_K(1).
$$
Since $\xi,\eta$ are $t$-independent, this holds uniformly in $\xi,\eta$.
Combining with $d_t(p,q)\le2\varepsilon_K(t)+d_t(\sigma(\xi),\sigma(\eta))$
from Lemma \ref{lem:fibre-collapse} gives
\[
d_t(p,q)\ \le\ d_G(\xi,\eta)+2\varepsilon_K(t)+o_K(1)\ =\ d_G(\bar\rho(p),\bar\rho(q))+o_K(1),
\]
since $\varepsilon_K(t)\to0$. This proves the claim. 
\end{proof}

\begin{theorem}
\label{thm:green-exhaustive-GH}
Let $S$ be a Kato surface of intermediate type, let $D$ be its maximal reduced divisor of rational curves, assume $0<\mu<2$, and fix $M>0$. On the open surface $S\setminus D$, consider the family of Hermitian metrics $\Omega^M_t$ defined in \eqref{eq:Omega-flow-def}.
Let $K=K_{s_0,R}$ be a compact Green block. Then
\begin{equation*}
(K,d_t|_{K\times K})
\longrightarrow
(S^1_G,d_G)
\end{equation*}
in the Gromov--Hausdorff sense as $t\to+\infty$.  Thus, along the exhaustion by compact Green blocks, the explicit normalised solution $\Omega_t^M$ collapses $S\setminus D$ to the Green circle.
\end{theorem}

\begin{proof}
By construction, $\bar\rho(K)=S^1_G$, so $\bar\rho:K\to S^1_G$ is onto. We show that its distortion tends to zero.
The lower estimate \eqref{eq:GH-lower-bound} gives
$$
d_t(p,q)
\geq
\left(\sqrt{\frac{A(t)}{2-\mu}}\right)
d_G(\bar\rho(p),\bar\rho(q)).
$$
Since $A(t)\to2-\mu$ and $S^1_G$ has finite diameter,
$$
d_t(p,q)
\geq
d_G(\bar\rho(p),\bar\rho(q))-o_K(1)
$$
uniformly for $p,q\in K$. The opposite estimate is
\Cref{lem:green-upper-estimate}:
$$
d_t(p,q)
\leq
d_G(\bar\rho(p),\bar\rho(q))+o_K(1).
$$
Hence
$$
\sup_{p,q\in K}
\left|
d_t(p,q)-d_G(\bar\rho(p),\bar\rho(q))
\right|
\longrightarrow0.
$$
Thus the distortion of $\bar\rho$ tends to zero. Since $\bar\rho$ is
onto, it is an $\eta_t$-Gromov--Hausdorff approximation with
$\eta_t\to0$. This proves the asserted Gromov--Hausdorff convergence.
\end{proof}

\subsection*{Acknowledgements}
 The first author is partially supported by project PRIN2022 ``Real and Complex Manifolds: Geometry and Holomorphic Dynamics'' (code 2022AP8HZ9), is a member of GNSAGA of INdAM, and thanks the Universit\`a degli Studi di Bari for its hospitality during several visits. The second author is partially supported by the Universit\`a degli Studi di Bari and by the PRIN 2022MWPMAB--``Interactions between Geometric Structures and Function Theories'', is a member of INdAM-GNSAGA, and thanks the Universit\`a degli Studi di Firenze for its hospitality during several visits since 2023, when this work was initiated.

\subsection*{AI disclosure}
AI were used for proofreading, improving the exposition, and checking computations. All mathematical content, arguments, and results are the authors' own; the authors have independently verified the final manuscript and assume full responsibility for its content.

\end{document}